\documentclass[journal,10pt,twocolumn]{IEEEtran}
\usepackage[mathscr]{eucal}
\usepackage{epsfig,epsf,psfrag}
\usepackage{amssymb,amsmath,amsfonts,latexsym}
\usepackage{graphicx}
\usepackage{epstopdf}
\usepackage{bm,xcolor,url}
\usepackage[caption=false]{subfig} 
\usepackage{fixltx2e}%ordering of single and double column floats
\usepackage{array}%array and tabular environments
\usepackage{verbatim}
\usepackage{bm}
\usepackage{algpseudocode, cite}
\usepackage{algorithm}
\usepackage{verbatim}
\usepackage{textcomp}
\usepackage{mathrsfs}
\usepackage{relsize}
\usepackage{subfig}
\usepackage{amsthm}
\usepackage{enumerate}
\usepackage{setspace}
\usepackage{cases}
\usepackage{footnote}
\usepackage{tabularx,tabu}
\usepackage{tablefootnote}
\usepackage{threeparttable}
\usepackage{tikz}
\usetikzlibrary{datavisualization}
\usetikzlibrary{datavisualization.formats.functions}
\usetikzlibrary{decorations.pathreplacing}
\usepackage{textcomp}
\usepackage{multicol}
\usepackage{lipsum}
\usepackage{enumitem}
\usepackage{bbm}
\usepackage{stackengine}
\stackMath 
\usepackage{empheq}
\catcode`~=11 \def\UrlSpecials{\do\~{\kern -.15em\lower .7ex\hbox{~}\kern .04em}} \catcode`~=13 

\allowdisplaybreaks[3]

\newcommand{\norm}[1]{\left\Vert#1\right\Vert}
\newcommand{\normt}[1]{\Vert#1\Vert}

\newcommand{\abs}[1]{\left\lvert#1\right\rvert}

\newcommand{\nn}{\nonumber}

\newcommand{\defeq}{\triangleq}
\newcommand{\eqcst}{\stackrel{{\rm c}}{=}}

\newcommand{\inter}{\mathbf{int}\,}

\newcommand{\barbH}{\overline{\bf H}}

\newcommand{\barbW}{\overline{\bf W}}
\newcommand{\barbX}{\overline{\bf X}}

\newcommand{\barell}{\overline{\ell}}

\newcommand{\calA}{\mathcal{A}}

\newcommand{\calK}{\mathcal{K}}

\newcommand{\calS}{\mathcal{S}}

\newcommand{\calX}{\mathcal{X}}

\newcommand{\bH}{\mathbf{H}}

\newcommand{\bM}{\mathbf{M}}

\newcommand{\bR}{\mathbf{R}}

\newcommand{\bS}{\mathbf{S}}

\newcommand{\bV}{\mathbf{V}}

\newcommand{\bW}{\mathbf{W}}
\newcommand{\bx}{\mathbf{x}}
\newcommand{\bX}{\mathbf{X}}

\newcommand{\bY}{\mathbf{Y}}

\newcommand{\bZ}{\mathbf{Z}}

\newcommand{\rmF}{\mathrm{F}}

\newcommand{\bbN}{\mathbb{N}}

\newcommand{\bbR}{\mathbb{R}}

\DeclareMathAlphabet{\mathbsf}{OT1}{cmss}{bx}{n}
\DeclareMathAlphabet{\mathssf}{OT1}{cmss}{m}{sl}% slanted sans serif

\DeclareSymbolFont{bsfletters}{OT1}{cmss}{bx}{n}  
\DeclareSymbolFont{ssfletters}{OT1}{cmss}{m}{n}
\DeclareMathSymbol{\bsfGamma}{0}{bsfletters}{'000}
\DeclareMathSymbol{\ssfGamma}{0}{ssfletters}{'000}
\DeclareMathSymbol{\bsfDelta}{0}{bsfletters}{'001}
\DeclareMathSymbol{\ssfDelta}{0}{ssfletters}{'001}
\DeclareMathSymbol{\bsfTheta}{0}{bsfletters}{'002}
\DeclareMathSymbol{\ssfTheta}{0}{ssfletters}{'002}
\DeclareMathSymbol{\bsfLambda}{0}{bsfletters}{'003}
\DeclareMathSymbol{\ssfLambda}{0}{ssfletters}{'003}
\DeclareMathSymbol{\bsfXi}{0}{bsfletters}{'004}
\DeclareMathSymbol{\ssfXi}{0}{ssfletters}{'004}
\DeclareMathSymbol{\bsfPi}{0}{bsfletters}{'005}
\DeclareMathSymbol{\ssfPi}{0}{ssfletters}{'005}
\DeclareMathSymbol{\bsfSigma}{0}{bsfletters}{'006}
\DeclareMathSymbol{\ssfSigma}{0}{ssfletters}{'006}
\DeclareMathSymbol{\bsfUpsilon}{0}{bsfletters}{'007}
\DeclareMathSymbol{\ssfUpsilon}{0}{ssfletters}{'007}
\DeclareMathSymbol{\bsfPhi}{0}{bsfletters}{'010}
\DeclareMathSymbol{\ssfPhi}{0}{ssfletters}{'010}
\DeclareMathSymbol{\bsfPsi}{0}{bsfletters}{'011}
\DeclareMathSymbol{\ssfPsi}{0}{ssfletters}{'011}
\DeclareMathSymbol{\bsfOmega}{0}{bsfletters}{'012}
\DeclareMathSymbol{\ssfOmega}{0}{ssfletters}{'012}

\newcommand{\tilf}{\widetilde{f}}
\newcommand{\tilF}{\widetilde{F}}

\newcommand{\tilG}{\widetilde{G}}

\newcommand{\tilh}{\widetilde{h}}

\newcommand{\tilbH}{\widetilde{\bH}}

\newcommand{\hatv}{\widehat{v}}

\newcommand{\hatbV}{\widehat{\bV}}

\newcommand{\tilw}{\widetilde{w}}

\newcommand{\tilbW}{\widetilde{\bW}}

\newcommand{\tilx}{\widetilde{x}}

\newcommand{\tilbZ}{\widetilde{\bZ}}

\newcommand{\barf}{\overline{f}}

\newcommand{\barD}{\overline{D}}

\newcommand{\barF}{\overline{F}}
\newcommand{\barG}{\overline{G}}

\newcommand{\tlambda}{\widetilde{\lambda}}

\newcommand{\floor}[1]{\lfloor{#1}\rfloor}
\newcommand{\lrangle}[2]{\left\langle{#1},{#2}\right\rangle}

\DeclareMathOperator*{\argmin}{arg\,min}

\DeclareMathOperator{\sgn}{sgn}

\newtheorem{theorem}{Theorem} 
\newtheorem*{theorem*}{Theorem}
\newtheorem{lemma}{Lemma}

\newtheorem{prop}{Proposition}

\theoremstyle{definition}
\newtheorem{definition}{Definition} 

\theoremstyle{remark}
\newtheorem{remark}{Remark}

\newcommand{\qednew}{\nobreak \ifvmode \relax \else
      \ifdim\lastskip<1.5em \hskip-\lastskip
      \hskip1.5em plus0em minus0.5em \fi \nobreak
      \vrule height0.75em width0.5em depth0.25em\fi}

\usepackage[colorlinks=true,linkcolor=black,citecolor=black,urlcolor=black]{hyperref}
\newcommand{\cbW}{\mathring{\bW}}
\newcommand{\cbH}{\mathring{\bH}}

\newcommand{\cbZ}{\mathring{\bZ}}

\usepackage{empheq}
\usepackage[autostyle]{csquotes}

\title{A Unified Convergence Analysis of the Multiplicative Update Algorithm for Regularized Nonnegative Matrix Factorization}
\author{
Renbo~Zhao,~\IEEEmembership{Student Member,}
        ~Vincent~Y.~F.~Tan,~\IEEEmembership{Senior~Member}
        
\thanks{
An abridged version of this paper was presented at the ICASSP 2017~\cite{Zhao_17c}. %\newline
Renbo~Zhao and Vincent~Y.~F.~Tan are with the Department of Electrical and Computer Engineering and the Department of Mathematics, National University of Singapore (NUS). Renbo Zhao is also with the Department of Industrial and Systems Engineering, NUS. They are supported in part by the NUS Young Investigator Award (grant number R-263-000-B37-133).}
}
 
\begin{document}

\maketitle
\begin{abstract}
The multiplicative update (MU) algorithm has been extensively used   to estimate the basis and coefficient matrices in  nonnegative matrix factorization (NMF) problems under a wide range of divergences and regularizers.  However, theoretical convergence guarantees have only been derived for a few special divergences without regularization.  
In this work, we provide a conceptually simple, self-contained, and unified proof for the convergence  of the MU algorithm applied on NMF with a wide range of divergences and regularizers.
Our main result shows the sequence of iterates (i.e., pairs of basis and coefficient matrices) produced by the MU algorithm converges to the set of stationary points of the non-convex NMF optimization   problem. Our proof strategy has the potential to open up new avenues for analyzing similar  problems in machine learning and signal processing. 
\end{abstract}


\section{Introduction}
%\textcolor{red}{Unify $:=$ and $\triangleq$} 

%In recent years, 
Nonnegative Matrix Factorization (NMF) has been a popular dimensionality reduction technique in recent years, due to its non-subtractive and parts-based interpretation on the learned basis \cite{Lee_99}. In the general formulation of NMF, given a nonnegative matrix $\bV\! \in \!\bbR_+^{F\times N}$, one seeks to find a nonnegative basis matrix $\bW \in \bbR_+^{F\times K}$  and a nonnegative coefficient matrix $\bH \!\in \!\bbR_+^{K\times N}$
such that $\bV\! \approx \!\bW\bH$. To find such  pair of matrices, a popular approach is to solve  the optimization problem
\begin{equation}
\min_{\bW\ge 0, \bH\ge 0} \left[\barell(\bW,\bH)\defeq\barD(\bV\Vert\bW\bH)\right].\label{eq:batch_NMF}
\end{equation}
%In \eqref{eq:batch_NMF}, 
where $\barD(\cdot\Vert\cdot)$ denotes the {\em divergence} (or distance) between two nonnegative matrices and $\bW \ge  0$ (and $\bH \ge  0$) denotes entrywise inequality. %Moreover, $\{\phi_i(\cdot)\}_{i=1}^{m_1}$ and $\{\tphi_j(\cdot)\}_{j=1}^{m_2}$ denote the (nonnegative) regularizers on the coefficient matrix and the basis matrix respectively, and their associated regularization weights are denoted by (nonegative) $\{\lambda_i\}_{i=1}^{m_1}$ and $\{\tlambda_j\}_{j=1}^{m_2}$ respectively. 
In the NMF literature, many algorithms have been proposed to solve \eqref{eq:batch_NMF}, including multiplicative updates (MU) \cite{Lee_00,Dhillon_06,Fev_11,Yang_11c}, block principal pivoting (BPP) \cite{Kim_08a}, projected gradient descent (PGD) \cite{Lin_07a} %, active set methods (ASM)\cite{Kim_08b} 
and the alternating direction method of multipliers (ADMM) \cite{Xu_12,Sun_14}. However, some algorithms only solve \eqref{eq:batch_NMF} for certain divergences $\barD(\cdot\Vert\cdot)$. For example, the BPP algorithm
%are devised to solve the problem of alternating nonnegative least squares (ANLS), thus they 
is only applicable to the squared-Frobenius loss $\|\bV-\bW\bH\|_{\mathrm{F}}^2$. Among all the algorithms, the MU algorithm has arguably  
%been applied to a wide range of divergences, 
the widest applicability---it has been used to solve~\eqref{eq:batch_NMF} when $\barD(\cdot\Vert\cdot)$ belongs to the family of $\alpha$-divergence \cite{Cichoc_08}, $\beta$-divergence \cite{Fev_11}, $\gamma$-divergence \cite{Cichoc_10}, etc. %See \cite{Yang_11c} for a comprehensive summary of these divergences.  
%Its wide applicability is due to its ease of implementation, intuitive interpretations and weak assumptions on the objective fu
%Its wide applicability is due to its ease of implementation, intuitive interpretations and weak assumptions on the objective function $\barell(\cdot,\cdot)$. 

Despite the popularity of the MU algorithm, its convergence properties {have not been studied systematically when the divergence is not the standard squared-Frobenius loss  and when there are regularizers on $\bW$ and $\bH$}. %\footnote{\textcolor{red}{VINCENT REMARK: I modified this because later on you say that ``To understand the theoretical convergence properties of this sequence, many works have been conducted'', which contradicts your original sentence}} 
To describe this problem precisely, let $\{(\bW^t,\bH^t)\}_{t=1}^\infty$ be the sequence of  pairs of basis and coefficient matrices  generated by the MU algorithm, where $t \ge  1$ denotes the iteration index.  
Many previous works \cite{Lee_00,Cichoc_08,Fev_11} showed  that the sequence of (nonnegative) objective values $\{\barell(\bW^t,\bH^t)\}_{t=1}^\infty$   in the MU algorithm is non-increasing and hence the algorithm converges. 
However, the convergence of objective values does not imply the convergence of %the sequence of matrix pairs 
$\{(\bW^t,\bH^t)\}_{t=1}^\infty$, whose limit points (assuming they exist) serve as natural candidates for the output of the MU algorithm.
%Moreover, when the MU algorithm is used on real applications, %such as music analysis \cite{Fev_09}, topic modeling \cite{Lee_99} and source separation \cite{Cichoc_08}, 
The limit points of $\{(\bW^t,\bH^t)\}_{t=1}^\infty$ tend to be empirically appealing, i.e., they represent each column of $\bV$ (i.e., a data sample) as a linear combination of $K$ nonnegative basis vectors in a meaningful manner~\cite{Lee_99,Fev_09}. % hence are  
%For example, in the topic modeling task~\cite{Greene_14}, the columns of the learned basis matrix correctly indicate the keywords in each underlying topic. 
As such, the convergence  properties of $\{(\bW^t,\bH^t)\}_{t=1}^\infty$, and especially the optimality of its limit points, are of   theoretical and  practical importance.
% \vspace{-.1in}
\subsection{Related Works} \label{sec:related} %\vspace{-.1in}
Due to the nonconvex nature of \eqref{eq:batch_NMF}, algorithms that guarantee to converge to the global minima of \eqref{eq:batch_NMF}
are in general out-of-reach. Indeed, \cite{Vavasis_09} has shown that~\eqref{eq:batch_NMF} is NP-hard. 
To ameliorate this situation, a line of works in which structural assumptions on the data matrix $\bV$---such as the separability~\cite{Donoho_04} assumptions~\cite{Bitt_12}---has emerged. Under such an assumption, polynomial-time algorithms~\cite{Arora_12,Bitt_12,Gillis_14,Huang_14b} have been proposed to find $\bW$ and $\bH$ such that $\bW\bH$ {\em exactly} equals $\bV$.  However, for many applications in  signal processing and machine learning, the data matrix $\bV$ does not strictly satisfy the  aforementioned assumptions. % is contaminated by noise, thereby making the assumptions leveraged in these works invalid.
In such scenarios, exact (nonnegative) factorization of $\bV$ is generally infeasible.

As a result, given a general nonnegative matrix $\bV$, many works~\cite{Lee_00,Dhillon_06,Fev_11,Yang_11c,Kim_08a,Lin_07a,Xu_12,Sun_14} only aim to reduce (or preserve) the function value of $\barell(\cdot,\cdot)$ at each iteration $t$, hoping that the sequence $\{(\bW^t,\bH^t)\}_{t=1}^\infty$ will converge to a limit point that is ``reasonably good''. 
To understand the theoretical convergence properties of this sequence, many works have been conducted~\cite{Lin_07c,Gillis_08,Taka_14,Taka_14b,Kim_07,Haji_16,Fev_09b}. 
%only aims to find the stationary points (see Definition~\ref{def:stat_pt}) of~\eqref{eq:batch_NMF}. Among these works, 
%Thus   existing works mainly study  convergence to the stationary points (see Definition~\ref{def:stat_pt}) of \eqref{eq:batch_NMF}. %\footnote{} 
In particular, for the MU algorithm, some representative works include \cite{Lin_07c,Gillis_08,Taka_14,Taka_14b}. %For simplicity, all of the MU algorithms in these works % focus on analyzing the algorithms proposed for solving 
%For a special case of \eqref{eq:batch_NMF}, namely 
When the divergence $\barD(\bV\Vert\bW\bH)=\frac{1}{2}\norm{\bV-\bW\bH}_\rmF^2$,   Lin~\cite{Lin_07c} and Gillis and Glineur~\cite{Gillis_08} modified the algorithm originally proposed in~\cite{Lee_00} (in different ways), and proved the convergence of the modified algorithms to the set of stationary points\footnote{See Definition~\ref{def:conv_set} for the definition of convergence of a sequence to a set.} of~\eqref{eq:batch_NMF}. However, their approaches cannot be easily generalized to other divergences, e.g., the (generalized) Kullback-Leibler (KL) divergence. To overcome this restriction, the authors of \cite{Taka_14} and \cite{Taka_14b} modified the nonnegativity constraints on $\bW$ and $\bH$ in~\eqref{eq:batch_NMF} to $\bW \ge  \epsilon$ and $\bH \ge  \epsilon$, for some $\epsilon > 0$. Accordingly, they developed algorithms for this ``positive matrix factorization'' problem~\cite{Paa_94} with a wider class of divergences (including the $\beta$-divergences) and showed that their algorithms converge  to the stationary points of the new problem. However, the positivity constraints on $\bW$ and $\bH$ are restrictive and changes the original NMF problem in~\eqref{eq:batch_NMF} {substantially}. Therefore the convergence analyses in \cite{Taka_14} and \cite{Taka_14b} are not applicable to the MU algorithms for the canonical NMF problem ({in which $\bW$ and $\bH$ are allowed to have {\em zero} entries}), which is of  interest in  many applications. %adopted a different analysis approach, which can be applied to a wider class of divergences (including the $\beta$-divergences). However, in these two works, the nonnegativity constraints on $\bW$ and $\bH$ in~\eqref{eq:batch_NMF} were modified to $\bW \ge  \epsilon$ and $\bH \ge  \epsilon$, for some $\epsilon > 0$. In other words, 
%\begin{equation}
%\min_{\bW\ge 0, \bH\ge 0} \frac{1}{2}\norm{\bV-\bW\bH}_F^2.\label{eq:sqFro_batch_NMF}
%\end{equation}
%, the state-of-the-art analysis was done by Lin~\cite{Lin_07c}. 
%In particular, a principled and rigorous analysis was performed in \cite{Lin_07c}. %the author considers the case where $\barD(\cdot\Vert\cdot)$ is the squared-Frobenius loss without  regularizations on $\bW$ or $\bH$. By modifying 
%In \cite{Lin_07c}, Lin modifies the MU algorithm proposed in \cite{Lee_00}, %for \eqref{eq:sqFro_batch_NMF}, 
%and shows the sequence of iterates $\{(\bW^k,\bH^k)\}_{k=1}^\infty$ generated by this algorithm converges to the set of stationary points\footnote{See Definition~\ref{def:conv_set} for the definition of convergence of a sequence to a set.}  of \eqref{eq:batch_NMF}. %The convergence analysis in \cite{Lin_07c} is . 
%%The basic idea therein is to interpret the MU algorithm as a gradient-descent algorithm (with step sizes being judiciously chosen) and leverage the KKT conditions to check the optimality of the limit points. In addition, Lin also provides a comprehensive survey on the analysis techniques prior to \cite{Lin_07c}. 
%Later, the authors of \cite{Gillis_08} propose different modifications of the MU algorithm in \cite{Lee_00} %(with the squared Frobenius loss) 
%and then provide sound convergence analyses accordingly. 
In another related work~\cite{Badeau_10}, the authors analyzed the stability of local minima of \eqref{eq:batch_NMF}
%the objective function $\barell(\cdot,\cdot)$ 
under the MU algorithm, where $\barD(\cdot\Vert\cdot)$ belongs to the class of $\beta$-divergences. However, the stability analysis therein does not yield definite answers on  whether (and when) the MU algorithm converges to any local minimum (or even stationary point) of $\barell(\cdot,\cdot)$ if the algorithm is started at an arbitrary (feasible) starting point. 
%For other algorithms that aim to solve \eqref{eq:batch_NMF}, some rigorous convergence analyses have been done in \cite{Kim_07,Haji_16,Fev_09b}. 
%However, all of the analyses are confined to some special cases of $\barD(\cdot\Vert\cdot)$, including the Itakura-Saito (IS), (generalized) Kullback-Leibler (KL) or squared-Frobenius losses.

\subsection{Motivations and Main Contributions}%\vspace{-.1in}
In this work, we analyze the convergence of the MU algorithm for {\em regularized} NMF problems with a general class of divergences, termed {\em $h$-divergences} (see Definition~\ref{def:general_div}) in a {\em unified} manner. The set of $h$-divergences includes many important classes of divergences, including (but not limited to)  $\alpha\,(\alpha\ne 0)$, $\beta$, $\gamma$, $\alpha$-$\beta$ and R\'enyi divergences. For each class of divergences, the corresponding MU algorithm has been proposed in the literature~\cite{Cichoc_08,Fev_11,Cichoc_11,Cichoc_10}, but without convergence guarantees. In addition, we also include regularizers on $\bW$ and $\bH$ in the objective function. The purpose of including regularizers are twofold: (i)   convenience of mathematical analysis and (ii) increased generality of problem setting. Although many MU algorithms have been proposed for NMF problems with various regularizers~\cite{Hoyer_02,Cai_11,Tasla_12,Mirzal_14}, thus far, the convergence analyses of these algorithms are still lacking. The absence of  theoretical  convergence guarantees for the MU algorithms in the abovementioned cases thus becomes a major motivation of our work. 

Our contributions consist of two parts. First, we develop a unified MU algorithm for the NMF problem~\eqref{eq:batch_NMF}  with any (weighted) $h$-divergence and $\ell_{1,1}$ (and Tikhonov) regularizers on $\bW$ and $\bH$. Our algorithm subsumes many existing algorithms in previous works~\cite{Lee_00,Dhillon_06,Fev_11,Yang_11c} as special cases. From our update rules, we discover that minimizing $\barD(\bV\Vert\bW\bH)$ with the $\ell_{1,1}$  regularization on $\bW$ and $\bH$ corresponds to a stability-preserving heuristic commonly employed in implementing the MU algorithms. (See Remark~\ref{rmk:MU} for details.) Therefore, this justifies the need to {incorporate $\ell_{1,1}$ regularization into the NMF objective}. Second, we conduct a novel convergence analysis for this unified MU algorithm, by making innovative use of the recently-proposed {\em block majorization-minimization} framework~\cite{Raza_13,Hong_16}. Our results show that the sequence of iterates $\{(\bW^t,\bH^t)\}_{t=1}^\infty$ generated from our MU algorithm has at least one limit point and any limit point of this sequence is a stationary point of \eqref{eq:batch_NMF}. Thus, for the first time,  it is shown that the host of  MU algorithms in the NMF literature~\cite{Cichoc_08,Fev_11,Cichoc_11,Cichoc_10,Hoyer_02,Cai_11,Tasla_12,Mirzal_14} enjoys  strong theoretical convergence guarantees. 

\subsection{Notations}%\vspace{-.1in}
In this paper we use $\bbR_+$, $\bbR_{++}$ and $\bbN$ to denote the set of nonnegative real numbers, positive real numbers and natural numbers (excluding zero) respectively. For $n\in\bbN$, we define $[n]\defeq \{1,2,\ldots,n\}$. 
%For any $k,l\in\bbN$, $\delta_{kl}$ denotes the Kronecker delta function. 
We use boldface capital letters, boldface lowercase letters and plain lowercase letters to denote matrices, vectors and scalars respectively. 
%For a nonnegative scalar $x$, we define its sign, $\sgn(x)$ as
%\begin{equation}
%\sgn(x)\defeq \left\{\hspace{-.2cm}\begin{array}{ll}
%1,&x>0\\
%0,&x=0
%\end{array}\right..
%\end{equation}
%For matrices $\bX$ and $\bY$, we use $\bX\odot\bY$, $\bX/\bY$ and $\lrangle{\bX}{\bY}$ to denote their Hadamard product, entrywise quotient and Frobenius inner product respectively.
For a vector $\bx$, we denote its $i$-th entry, $\ell_1$ and $\ell_2$ norms as $x_i$, $\norm{\bx}_1$ and $\norm{\bx}_2$ respectively. 
For a matrix $\bX$, we denote %its its $i$-th row as $\bx^i$ and %$j$-th column as $\bx_j$,  
its $(i,j)$-th entry as $x_{ij}$ %(or $\left[\bX\right]_{ij}$)   
and its $\ell_{1,1}$ norm as $\norm{\bX}_{1,1}\defeq \sum_{ij} \abs{x_{i,j}}$. 
In addition, for a scalar $\delta \in \bbR$, %we use $\bX\ge 0$ and $\bX>0$ to denote entrywise nonnegativity and positivity respectively. 
we use $\bX=\delta$ and $\bX\ge \delta$ to denote entrywise equality and inequality. 
For matrices $\bX$ and $\bY$, we use $\bX\odot\bY$ and $\lrangle{\bX}{\bY}$ to denote their Hadamard product and Frobenius inner product respectively.
%For a set $\calS$, we use $\inter\calS$, $\bdr\calS$ and $\cl\calS$ to denote its interior, boundary and closure respectively. 
We use $\eqcst$ to denote equality up to additive constants. %(whose meanings depend on the context). 
%We also use $\vecz$ and $\vecone$ to denote the vector or matrix with all its entries equal to zero and one respectively. The dimensions of $\vecz$ and $\vecone$ depend on the context. 
In this work, technical lemmas (whose indices begin with `T') will appear in Appendix~\ref{app:tech_lemma}.  

\section{Problem Formulation} \label{sec:problem_form}
\subsection{Definition of $h$-Divergences}
Before introducing the notion of {\em $h$-divergences}, we first define an important function 
\begin{equation} 
h(\sigma,t)\defeq \left\{\hspace{-.2cm}\begin{array}{ll}
(\sigma^t-1)/t, &t\ne 0\\
\log\sigma, & t=0
\end{array}\right., \label{eq:def_h}
\end{equation}
where for any $t \in \bbR$, the domain of $\sigma$ is given by the natural domain of $\sigma \mapsto h(\sigma,t)$, denoted as $\Xi_t$.  To be more explicit, $\Xi_t=[0,\infty) $ if $t\ne 0$ and $\Xi_t=(0,\infty)$ if $t=0$. 

\begin{definition}[$h$-divergences~{\cite[Section~IV]{Yang_11c}}]\label{def:general_div}
Given any $\bV \in \bbR_{+}^{F \times  N}$, $D(\bV\Vert\cdot):\bbR_{+}^{F \times  N} \to \bbR_+$ is called a $h$-divergence if for any $\hatbV \in \bbR_{+}^{F \times  N}$ , there exists a constant $P \ge  2$ %$\{\mu_p\}_{p\in[P]}$, $\{\zeta_p\}_{p\in[P]}$, $\{\xi_p\}_{p\in[P]}$ and $\{\nu_{pij}\}_{(p,i,j)\in[P]\times[F]\times[N]}\subseteq\bbR$ (all independent of $\hatbV$), 
such that
\begin{equation}
\hspace{-.2cm}D(\bV\Vert\hatbV)\eqcst\sum_{p=1}^P\mu_p h\left(\sum_{i=1}^F\sum_{j=1}^N\nu_{pij}h(\hatv_{ij},\zeta_p),\,\xi_p\right),\label{eq:general_div}
\end{equation}
where `$\eqcst$' omits constants that are independent of $\hatbV$ and $\mu_p$, $\nu_{pij}$, $\zeta_p$ and $\xi_p$ are all real constants independent of $\hatbV$. In addition, $\{\zeta_p\}_{p=1}^P$ are distinct and for any $i \in [F]$ and $j \in [N]$, there exists $p' \in [P]$ such that $\nu_{p'ij}\ne 0$. 
%In addition, for any $\bV\in\bbR_{++}^{m\times n}$, there exists a $\bV^*\ne \vecz$ such that $\nabla_{\hatbV} D(\bV\Vert\hatbV)\big\vert_{\hatbV=\bV^*} = \vecz$.
%In \eqref{eq:general_div}, %We also denote the set of all the $h$-divergences as $\calD$. 
\end{definition}

\begin{remark}[Scope of $h$-divergences]\label{rmk:general_div}
%Several remarks are in order.
By choosing the constants  $\mu_p$, $\nu_{pij}$, $\zeta_p$  and $\xi_p$ in different ways, we obtain different $h$-divergences. 
The $h$-divergences subsume many important classes of divergences, including the families of $\alpha\,(\alpha \ne  0)$, $\beta$, $\gamma$, $\alpha$-$\beta$ and R\'enyi divergences~\cite{Cichoc_08,Fev_11,Cichoc_11,Cichoc_10}.  In particular, some important instances in the $h$-divergences include the Hellinger, Itakura-Saito (IS), KL and squared-Frobenius divergences. Each instance can be obtained by appropriately choosing the  constants $\{\mu_p\}_p$, $\{\nu_{pij}\}_{p,i,j}$, $\{\zeta_p\}_{p}$ and $\{\xi_p\}_{p}$.  
See Remark~\ref{rmk:sep_hDiv} for an example of how~\eqref{eq:general_div} yields the KL divergence with an appropriate set of parameters.
%\textcolor{red}{this is a bit obscure. I think we can have a bit more elaboration here on how various settings of $\mu$, $\nu$, ... results in these known divergences--for reviewers who may not be acquainted to NMF lit. In \eqref{eq:general_div} can you also state which def of \cite{Yang_11c} you're using? I couldn't find it. i don't see the ``nested'' $h$ in his paper, whereas your nested $h$ in \eqref{eq:general_div} can be quite confusing to the unacquainted. }
\end{remark}
%All of these divergences have appeared in the NMF literature \cite{Cichoc_08,Fev_11,Cichoc_11,Cichoc_10}.
\begin{remark}[Separable $h$-divergences]\label{rmk:sep_hDiv}
When $\mu_p=\xi_p=1$, for all $p\in[P]$, $D(\bV\Vert\cdot)$ is separable across the entries of $\hatbV$, i.e.,
\begin{equation}
D(\bV\Vert\hatbV)\eqcst\sum_{i=1}^F\sum_{j=1}^N\sum_{p=1}^P\nu_{pij}h(\hatv_{ij},\zeta_p). %\label{eq:sep_div}
\end{equation}
%In the sequel, 
We term such a divergence as a {\em separable $h$-divergence}. In particular, any member in the classes of $\alpha$- (for $\alpha\ne 0$)  or $\beta$-divergences is separable. For example, taking $P=2$, $\nu_{1ij}=-v_{ij}$, $\zeta_1=0$, $\nu_{2ij}=1$ and $\zeta_2=1$, we obtain the KL divergence, which belongs to both classes ($\alpha$- and $\beta$-divergences). %the $\alpha$- and $\beta$-divergence families. 
%Third, we require both $\bV$ and $\hatbV$ to lie in the {\em positive} orthant to avoid division by zero and taking logarithm w.r.t.\ zero. As will be seen in Algorithm~\ref{algo:general}, the (entrywise) positivity of $\hatbV$ at each iteration can be guaranteed. To ensure the positivity of the data matrix $\bV$, in practice, it is common to set zero entries of $\bV$ to small positive numbers in the data-preprocessing phase.
\end{remark}
\begin{remark}[Weighted $h$-divergences]
For some special instances in the class of $h$-divergences, such as squared Euclidean distance and KL divergence, a weighted version has been proposed and studied in the literature~\cite{Guill_01,Guill_03,Ho_08,Gu_10,Zheng_15}. Based on Definition~\ref{def:general_div}, we can also define {\em weighted} $h$-divergences, which subsume the aforementioned weighted divergences as special cases. Given a nonnegative matrix $\bM \in \bbR_{+}^{F \times N}$, define its support $\Omega(\bM) \defeq \{(i,j) \in[F] \times [N]:m_{ij} > 0\}$. Also, for any $p \in [P]$, define $\nu'_{pij}\defeq \nu_{pij}m_{ij}$. For any $h$-divergence $D(\bV\Vert\cdot)$, define its $\bM$-weighted version $D_{\bM}(\bV\Vert\cdot) : \bbR^{F\times N} \mapsto \bbR_+$ as  
\begin{equation}
\hspace{-.2cm}D_{\bM}(\bV\Vert\hatbV)\eqcst\sum_{p=1}^P\mu_p h\left(\sum_{(i,j)\in\Omega(\bM)}\nu'_{pij}h(\hatv_{ij},\zeta_p),\,\xi_p\right).\label{eq:weighted_div}
\end{equation}
%\textcolor{red}{ The above $D$ should be $D_{\bM}$ right? I changed for u}
Comparing \eqref{eq:weighted_div} to~\eqref{eq:general_div}, we observe that the only changes are the constants $\{\nu_{pij}\}_{pij}$. These constants are independent of $\hatbV$. Therefore, our algorithms and convergence analysis developed for the $h$-divergences are also applicable to their weighted counterparts. Indeed, the weighted $h$-divergences in~\eqref{eq:weighted_div} are more general than $h$-divergences in~\eqref{eq:general_div}, since by choosing $\bM$ such that $m_{ij}=1$ for any $i$ and $j$, we recover~\eqref{eq:general_div} from~\eqref{eq:weighted_div}. 
\end{remark}
\subsection{Optimization Problem}
%In this work,
For convenience, first define two functions $\phi_1(\cdot)  \defeq \norm{\cdot}_{1,1}$ and $\phi_2 (\cdot)  \defeq  \norm{\cdot}_\rmF^2$.  The first and second  functions are known as the $\ell_{1,1}$ and Tikhonov regularizers respectively.
Accordingly, for any $\bV\in \bbR_{+}^{F\times N}$, define the regularized objective  function
\begin{equation}
\ell(\bW,\bH) \defeq D(\bV\Vert\bW\bH) + \sum_{i=1}^2 \lambda_i \phi_i(\bW) + \sum_{j=1}^2\tlambda_j \phi_j(\bH),\label{eq:new_obj}
\end{equation}
where $\bW\!\in\!\bbR_{+}^{F\!\times \!K}$, $\bH\!\in\!\bbR_{+}^{K\!\times \!N}$, $\lambda_1,\tlambda_1\!>\!0$ and $\lambda_2,\tlambda_2\!\ge\! 0$. 
The optimization problem  in this work can be stated {succinctly} as 
\begin{equation}
\min_{\bW\in\bbR_{+}^{F\times K}, \bH\in\bbR_{+}^{K\times N}}\ell(\bW,\bH). \label{eq:new_opt}
\end{equation}
%where %$K<\min(F,N)$ and 

%In \eqref{eq:new_obj}, $\bV\in \bbR_{+}^{F\times N}$, %$\{\lambda_i\}_{i\in[2]}$, $\{\tlambda_j\}_{j\in[2]}\subseteq \bbR_+$ 
%$\lambda_1,\tlambda_1>0$ and $\lambda_2,\tlambda_2\ge 0$. %and for any nonnegative matrix $\bX$,
%\begin{equation}
%$\phi_1(\bX) \defeq \norm{\bX}_{1,1}\defeq \sum_{i,j} x_{ij}$ and  $\phi_2(\bX) \defeq \norm{\bX}_F^2$.
%\end{equation}

\begin{remark}[Explanations for the elastic-net regularization]
The above regularizers involving $\phi_1(\cdot)$ and $\phi_2(\cdot)$ are collectively known as the {\em elastic-net regularizer}~\cite{Zou_05}  in the literature. 
%includes   the $\ell_{1,1}$ and Tikhonov regularizers as special cases, 
%We explain why we focus on the so-called {\em elastic-net regularizer}~\cite{Zou_05} on $(\bW,\bH)$. %are termed the elastic-net regularizer in the literature . 
%This regularizer includes the $\ell_{1,1}$ and Tikhonov regularizers as special cases, 
Both $\ell_{1,1}$ and Tikhonov regularizers have been widely employed in the NMF literature. Specifically, the $\ell_{1,1}$ regularizer promotes element-wise sparsity on   the basis matrix $\bW$ and coefficient matrix $\bH$~\cite{Hoyer_02}, thereby enhancing  the interpretability of both basis vectors and the %(nonnegative) 
conic combination model in NMF.  
The Tikhonov regularizer promotes smoothness on $\bW$ and $\bH$ and also prevents overfitting~\cite{Pauca_06}. 
\end{remark}
\begin{remark}[Positivity of $\lambda_1$ and $\tlambda_1$]\label{rmk:pos_lambda}
We require both $\lambda_1$ and $\tlambda_1$ to be positive for both convenience of analysis and numerical stability. Specifically, the inclusion of $\ell_{1,1}$ regularization on $\bW$ and $\bH$ ensures that  both $\bW \mapsto \ell(\bW,\bH)$ and $\bH \mapsto \ell(\bW,\bH)$ coercive, a property that we will leverage in our analysis. In addition, as will be shown in Proposition~\ref{proposition:deriv_MU}, the positivity of $\lambda_1$ and $\tlambda_1$ prevents the denominators in the multiplicative update rules of $\bW$ and $\bH$  from being arbitrarily close to zero, thereby ensuring that the updates in the MU algorithm are  numerically stable.  %originates from a commonly used heuristic in the MU algorithm that ensures numerical stability in the updates. %Specifically, when these two parameters are very small, minimizing \eqref{eq:new_obj} amounts to adding a small positive number in the denominator of the multiplicative factor. 
%See Remark~\ref{rmk:MU} for details. 
%Also note that in such case, the effect of such regularizers on sparsity of $\bW$ or $\bH$ is negligible. 
%on the coefficient vectors and the basis matrices respectively. 
%Any of the regularizers can be eliminated by simply setting the corresponding regularization weight to zero. 
\end{remark}

\section{Algorithms}

In this section we first define the notions of {surrogate functions} and {first-order surrogate functions}. 
Next, we present a general framework for deriving the MU algorithm for the problem~\eqref{eq:new_opt}, based on {\em majorization-minimization}~\cite{Yang_11c}. This framework is sufficient for our convergence analysis. However, as  side contributions, we also present a {\em systematic} procedure to construct   first-order surrogate functions of $(\bW,\bH) \mapsto \ell(\bW,\bH)$ for $\bW$ (resp.\ $\bH$) and to derive the specific multiplicative update rules for  $\bW$ (resp.\ $\bH$).  {Finally, we discuss how to apply these techniques to the family of $\alpha$-divergences and how to extend the techniques to the dual KL divergence.} %These update rules are based on %
%minimizing . %Finally, based on these functions, we derive the 
%In this section we first introduce the notion of {\em first-order surrogate function}. Based on it, we present our framework of deriving multiplicative updates for the $h$-divergences.  Next, we introduce an effective method to construct the first order surrogate functions for the $h$-divergences. An example is given in terms of the $\alpha$-divergence. 
 %This class of functions will play an important role in the convergence analysis (shown in Section~\ref{sec:conv_analysis}).
\subsection{First-Order Surrogate Functions}
\begin{definition}\label{def:first_order_surrogate}
Given $n$ finite-dimensional real Euclidean spaces $\{\calX_i\}_{i=1}^n$, define $\calX \defeq \prod_{i=1}^n\calX_i$. For any $x \in \calX$, denote its $n$-block form as $(x_1,\ldots,x_n)$, where for any $i \in [n]$,  $x_i \in \calX_i$ denotes the $i$-th block of $x$. %  is the $i$-th block %(subvector) 
%of $x$. %(hence $m=\sum_{i\in[n]}m_i$). 
Consider a differentiable function $f : \calX \to \bbR$. For any $i \in [n]$, a {\em first-order surrogate function} of $(x_1,\ldots,x_n) \mapsto f(x_1,\ldots,x_n)$ for %the $i$-th block 
$x_i$, %at $\tilx\in\bbR_+^m$, 
denoted as $F_i(\cdot\,|\,\cdot) : \calX_i \times \calX \to \bbR$, satisfies the following five properties: 
\begin{enumerate}[label={\bf (P\arabic*)},ref=\textbf{(P\arabic*)}]
\item \label{prop:zero_eq} $F_i(\tilx_i\,|\,\tilx) = f(\tilx)$, for any $\tilx\in\calX$, 
\item \label{prop:zero_ge} $F_i(x_i\,|\,\tilx)\ge f(\tilx_1,\ldots,x_i,\ldots,\tilx_n)$, for any $(x_i,\tilx) \in \calX_i \times \calX$. %$x_i\in\calX_i$ and $\tilx\in\calX$, 
\item \label{prop:mu} $F_i(\cdot\,|\,\cdot)$ is differentiable on $\calX_i\times\calX$ and for any $\tilx\in\calX$, there exists a function $g(\cdot\,|\,\tilx) : \calX_i \to \bbR$  such that $\nabla F_i(\cdot\,|\,\tilx)= g(\cdot/\tilx_i\,|\,\tilx)$ on $\calX_i$. %, for any $x_i\in\calX_i$,
\item \label{prop:first_eq} %$\nabla_{x_i} F_i(x_i\,|\,\tilx)\vert_{x_i=\tilx_i}=\nabla f(\tilx)$, for any $\tilx\in\calX$,
 $\nabla_{x_i} F_i(x_i\,|\,\tilx)\vert_{x_i=\tilx_i}  = \nabla_{x_i} f(\tilx_1,\ldots,x_i,\ldots,\tilx_n)\vert_{x_i=\tilx_i}$, for any $\tilx\in\calX$,
\item \label{prop:strict_convex} $F_i(\cdot\,|\,\tilx)$ is strictly convex on $\calX_i$,  for any $\tilx\in\calX$.
\end{enumerate}
If $F_i(\cdot\,|\,\cdot)$ only satisfies properties \ref{prop:zero_eq} to \ref{prop:mu}, then it is called a {\em surrogate function} of $f$ for $x_i$. Note that in general, (first-order) surrogate functions {\em may not be unique}. 
%For any $i\in[n]$ and any $\tilx\in\calX$, we also define
%\begin{equation}
%\calF_i(\tilx)\defeq \{F(\cdot\,|\,\tilx)\,|\,F(\cdot\,|\,\tilx)\mbox{ satisfies \ref{prop:zero_eq} to \ref{prop:strict_convex}}\}.
%\end{equation}
\end{definition}
\begin{remark}[Implications of properties \ref{prop:zero_eq} to \ref{prop:strict_convex}]
%We now explain the implications of the five properties in Definition~\ref{def:first_order_surrogate}. 
%First, 
From \ref{prop:strict_convex}, we know the minimizer of $F_i(\cdot\vert\tilx)$ over $\calX_i$ is unique. Let us denote it as $x^*_i$. 
%\begin{align}
%x^*_i & \defeq \argmin_{x_i\in\calX_i} F_i(x_i\vert\tilx),\label{eq:min_Fi}
%\end{align} where the uniqueness of the minimizer in \eqref{eq:min_Fi} is guaranteed by \ref{prop:strict_convex}. 
From both \ref{prop:zero_eq} and \ref{prop:zero_ge}, we can deduce that $f(\tilx_1,\ldots,x_i^*,\ldots,\tilx_n)\le f(\tilx)$.
%Moreover, define $x^*  \defeq (\tilx_1,\ldots,x^*_i,\ldots,\tilx_n)$, then \ref{prop:zero_eq} and \ref{prop:zero_ge} together ensure $f(x^*)\le f(\tilx)$. 
In addition, \ref{prop:mu} ensures that minimizing $x_i \mapsto F_i(x_i\,\vert\,\tilx)$ over $\calX_i$ yields a multiplicative update for the $i$-th block $x_i$. %(see Section~\ref{sec:deriv_MU}). 
Finally, \ref{prop:first_eq} ensures that for any $\tilx \in \calX$ and $i \in [n]$, the gradient of $x_i \mapsto F_i(x_i\,\vert\,\tilx)$ agrees with that of $x_i \mapsto f(\tilx_1,\ldots,x_i,\ldots,\tilx_n)$ at $\tilx_i$. This property will be leveraged in our convergence analysis. 
%justifies the term ``first-order'', and its implication will be seen in the proof of Theorem~\ref{thm:main}. 
\end{remark}

\begin{remark}[Constant difference]
With a slight abuse of terminology, we shall term any function $x_i \mapsto \tilF_i(x_i\,\vert\,\tilx)$ a (first-order) surrogate function if it differs from $x_i \mapsto F_i(x_i\,\vert\,\tilx)$ by a constant that is independent of $x_i$. This is because such a constant difference does not affect the minimizer(s) of $F_i(\cdot\vert \tilx)$ or $\tilF_i(\cdot\vert \tilx)$ over $\calX_i$ or their gradients w.r.t.\ $x_i$. %$\min_{x_i\in\calX_i}f()$
Consequently, it does not affect the resulting multiplicative updates for $x_i$ or the convergence analysis in Section~\ref{sec:conv_analysis}. 
\end{remark}

\subsection{General Framework for Multiplicative Updates}
The general framework for deriving the MU algorithm for~\eqref{eq:new_opt} is shown in Algorithm~\ref{algo:general}, where 
%In \eqref{eq:upd_W} and \eqref{eq:upd_H}, 
%for any $t\in\bbN$, given any nonnegative $\tilbW$ and $\tilbH$, 
$G_1(\cdot\vert\cdot)$ and $G_2(\cdot\vert\cdot)$ denote the first-order surrogate functions of $(\bW,\bH) \mapsto \ell(\bW,\bH)$ %at $(\tilbW,\tilbH)$ 
for $\bW$ and $\bH$ respectively. %Specifically, under Definition~\ref{def:first_order_surrogate}, $\calX=\bbR_{++}^{F\times K}\times \bbR_{++}^{K\times N}$, $x=(\bW,\bH)$ and $f=\ell$.
As will be shown in Proposition~\ref{proposition:deriv_MU}, the minimization steps in~\eqref{eq:upd_W} and \eqref{eq:upd_H} indeed result in multiplicative updates for $\bW$ and $\bH$ respectively. 

\begin{algorithm}[t]
\caption{General Framework for Multiplicative Updates} \label{algo:general}
\begin{algorithmic} 
\State {\bf Input}: Data matrix $\bV \in \bbR^{F \times N}$, latent dimension $K$, regularization weights $\lambda_1,\tlambda_1 > 0$ and $\lambda_2,\tlambda_2 \ge  0$, maximum number of iterations $t_{\max}$%and divergence $D(\cdot\Vert\cdot)\in\calD$ 
%maximum number of iterations $T_{\rm max}$
\State {\bf Initialize} %basis matrix 
$\bW^0 \in \bbR_{++}^{F \times K}$, %coefficient matrix 
$\bH^0 \in \bbR_{++}^{K \times N}$ %iteration index $t:=0$
\State {\bf For} $t=0,1,\ldots,t_{\max}-1$
\begin{align}
%&\hspace{-2cm}\mbox{Construct a first-order surrogate function of } D(\bV\Vert\cdot) \mbox{ for } U_1(\cdot|) \\
\bW^{t+1} &:= \argmin_{\bW\in\bbR_{+}^{F\times K}} G_1(\bW\vert\bW^t,\bH^t)\label{eq:upd_W}\\
\bH^{t+1} &:= \argmin_{\bH\in\bbR_{+}^{K\times N}} G_2(\bH\vert\bW^{t+1},\bH^t)\label{eq:upd_H}
%t&:=t+1
\end{align}
%$\del\abs{\ell(\bW^t,\bH^t) -\ell(\bW^{t+1},\bH^{t+1})}$
%\State {\bf Until} some convergence criterion is met
\State {\bf End}
\State {\bf Output}: Basis matrix $\bW^{t_{\max}}$ and coefficient matrix $\bH^{t_{\max}}$
\end{algorithmic}
\end{algorithm}

\subsection{Construction of First-Order Surrogate Functions} % and Derivation of Multiplicative Updates}
\label{sec:construct_first_order}
% Proposition 1
We only focus on constructing a first-order surrogate function of $(\bW,\bH) \mapsto \ell(\bW,\bH)$ for $\bW$, and when $D(\bV\,\Vert\,\cdot)$ is a {\em separable} $h$-divergence. By symmetry between $\bW$ and $\bH$, such a surrogate function for $\bH$ can be easily obtained (by taking transposition). In addition,  since $h(\cdot,t)$ is either convex or concave for $t \in \bbR$ (cf.\ Lemma~\ref{lem:reg_h}), $D(\bV\,\Vert\,\cdot)$ in~\eqref{eq:general_div} is a {\em difference-of-convex function}. Therefore, by using either Jensen's inequality or a first-order Taylor expansion, the nonseparable $h$-divergence can be easily converted to the separable one. Such conversion techniques are common in the NMF literature~\cite{Yang_11c,Fev_11,Mairal_15}.

\begin{prop}%[Construction of the first-order function]
\label{proposition:construction} %for \eqref{eq:new_obj} with separable $D(\cdot\Vert\cdot)$; ]
%Let $\bV\in\bbR_{+}^{F\times N}$ and $D(\bV\Vert\cdot)$ be a separable $h$-divergence. 
For any $\bV\in\bbR_{+}^{F \times  N}$, $\bW \in \bbR_+^{F \times K}$, $\tilbW \in \bbR_{++}^{F \times K}$, $\tilbH \in \bbR_{++}^{K \times N}$, and $\vartheta_{1},\vartheta_{2} \in \bbR$, define $\tilbZ \defeq (\tilbW,\tilbH)$ and a function
\begin{align}
&\hspace{-.1cm}G_1(\bW\vert\tilbZ)\defeq \sum_{i=1}^F\sum_{k=1}^K \left[(s_{ik}^++\lambda_1)\tilw_{ik}h\left(\frac{w_{ik}}{\tilw_{ik}},\vartheta_{2}\right)\right.\nn\\
&\quad\; \left.+2\lambda_2\tilw_{ik}^2h\left(\frac{w_{ik}}{\tilw_{ik}},\vartheta_{2}\right)- s_{ik}^- \tilw_{ik}h\left(\frac{w_{ik}}{\tilw_{ik}},\vartheta_{1}\right)\right],\label{eq:first_order_surr_W}
\end{align}
where $\bS^+$ and $\bS^-$ %(both in $\bbR_{+}^{F\times K}$) are defined 
are the sums of positive and unsigned negative terms  in $\nabla_\bW D(\bV\Vert\bW\tilbH)\big\vert_{\bW=\widetilde{\bW}}$ respectively (cf.~\cite{Lee_00}).\footnote{Note that the decomposition of $\nabla_\bW D(\bV\Vert\bW\tilbH)\big\vert_{\bW=\widetilde{\bW}}$ into the positive and negative terms is not unique. However, Proposition~\ref{proposition:construction} (and hence Proposition~\ref{proposition:deriv_MU}) holds for any of such decompositions.} %Also let $\zeta_{\max}$, $\zeta_{\min}$, $\bS^+$ and $\bS^-$ be given as in Lemma~\ref{lem:surr_sep}.
%For any $\tilbW\in\bbR_{++}^{F\times K}$ and $\tilbH\in\bbR_{++}^{K\times N}$, define
%\begin{equation}
%$\zeta'_{\max} \defeq  \max\{\zeta_{\max},2\sgn(\lambda_2)\}$, %\label{eq:zeta_max_prime}
%\end{equation}
Then for any separable $h$-divergence $D(\bV\Vert\cdot)$, there exist real numbers $\vartheta_{1}<\vartheta_{2}$ %\in\bbR$ %, $\zeta_{\min}\le 1 \le \zeta_{\max}$ 
% , $\zeta_{\min} < \zeta_{\max}$, 
such that $G_1(\cdot\vert\cdot)$   %a first-order surrogate functon of 
is a first-order surrogate function of  $(\bW,\bH) \mapsto \ell(\bW,\bH)$ with respect to the variable~$\bW$. %(in $\bbR_{++}^{F\times K}$) 
%at $(\tilbW,\tilbH)$
 %(up to an additive constant that is independent of $\bW$).\footnote{Since the constant does not affect the minimizer of a first-order surrogate function over $\bW \in \bbR_{+}^{F \times K}$, in the sequel we will omit the words in the parenthesis for brevity.} %Here $\bS^+$ and $\bS^-$ (both in $\bbR_{+}^{F\times K}$) are defined as the sums of positive and unsigned negative terms  (cf.~\cite{Lee_00}) in $\nabla_\bW D(\bV\Vert\bW\tilbH)\big\vert_{\bW=\widetilde{\bW}}$ respectively.% (\textcolor{red}{ref for positive and negative terms of ...}).
\end{prop}
\begin{proof}
See Appendix~\ref{app:proof_construct}. 
\end{proof}

\subsection{Derivation of Multiplicative Updates}
Based on Proposition~\ref{proposition:construction}, by setting $\nabla_\bW G_1(\bW\vert\tilbZ)$ to zero, we obtain the multiplicative update for $\bW$. %corresponding to~\eqref{eq:upd_W} %(for $\bW$)
% in Algorithm~\ref{algo:general}. %for any separable $h$-divergence (with elastic-net regularization).\footnote{The multiplicative updates for $\bH$ can be similarly derived by taking transposition and switching the roles of $\bW$ and $\bH$.}
% Proposition 2
\begin{prop}%[Derivation of multiplicative updates]
\label{proposition:deriv_MU}
Let $\bV$, $D(\bV\Vert\cdot)$, $\vartheta_{1}$, $\vartheta_{2}$, $\bS^+$ and $\bS^-$ be given as in Proposition~\ref{proposition:construction}.
For any $t\ge 0$, let $\tilbW$ (resp.\ $\tilbH$) denote the basis (resp.\ coefficient) matrix at iteration  $t$, and $\bW$ (resp.\ $\bH$) denote the basis (resp.\ coefficient) matrix at iteration  $t+1$.
For any $(i,k) \in [F] \times [K]$, the  multiplicative update corresponding to \eqref{eq:upd_W} in Algorithm~\ref{algo:general} admits the form %\footnote{Here }
\begin{equation}
w_{ik} := \tilw_{ik}\left(\frac{s^-_{ik}}{s^+_{ik}+2\lambda_2\tilw_{ik}+\lambda_1}\right)^{{1}/{(\vartheta_{2}-\vartheta_{1})}}.%,\;(i,k)\in[F]\times[K]
\label{eq:MU_W}
\end{equation}
%where $w_{ik}$ denotes the update of the  original iterate $\tilw_{ik}$.
%for any $(i,k)\in[F]\times[K]$.
\end{prop}

\begin{remark}[Numerical Stability]\label{rmk:MU}
%We make two remarks about \eqref{eq:MU_W}. First, we notice the multiplicative factor (i.e., the right-hand side of \eqref{eq:MU_W} without $\tilw_{ik}$) is always positive. This means if the initial basis matrix $\bW_0\in\bbR_{++}^{F\times K}$,\footnote{The typical random initialization method used in the literature of NMF initializes each entry of $\bW_0$ \iid from a continuous distribution (with nonnegative real support). Using this method, $\bW^0\in\bbR_{++}^{F\times K}$ almost surely.} then at any iteration $t\in\bbN$, $\bW^t\in\bbR_{++}^{F\times K}$. In other words, the sequence $\{\bW^t\}_{t=1}^\infty$ can only be {\em asymptotically sparse}. However, this poses no problem in practice since for sufficiently large $t$, the sparsity pattern of $\bW^t$ can be recovered simply by thresholding. Second, 
In \eqref{eq:MU_W}, the presence of $\lambda_1 > 0$ ensures numerical stability, i.e., it prevents 
the denominator of the multiplicative factor to be arbitrarily small 
%division by extremely small numbers 
(which may lead to numerical overflow). As a popular heuristic (e.g., \cite{Cichoc_11}), a small positive number is usually added to this denominator artificially. Here we \emph{establish the connection} between this artificially added small number and the $\ell_1$ regularization for {$h$-divergences}, thereby theoretically justifying this heuristic.\footnote{This connection has been  observed for some special $h$-divergences\cite{Hoyer_02,Fev_11}, but here we provide a more general and unified discussion.}
\end{remark}

\begin{remark}[Positivity of $\bW$]
As shown in~\cite{Yang_11c}, both matrices $\bS^+$ and $\bS^-$ are   entry-wise positive, i.e., $\bS^+,\bS^- \in \bR_{++}^{F \times  K}$. Therefore, if $\tilbW \in \bbR_{++}^{F \times  K}$ in~\eqref{eq:MU_W}, then $\bW \in \bbR_{++}^{F \times  K}$. Since the initial basis matrix $\bW^0 \in \bbR_{++}^{F \times  K}$ in Algorithm~\ref{algo:general}, for any {\em finite} index $t \in \bbN$, $\bW^t$ will be entry-wise positive. Similar arguments apply to $\bH^t$. Therefore, %we can always construct a well-defined first-order surrogate function
the positivity requirements on $(\tilbW,\tilbH)$ in Proposition~\ref{proposition:construction} can be satisfied at any finite iteration. Note that this does not prevent {\em any limit point} of $\{(\bW^t,\bH^t)\}_{t=1}^\infty$ from having zero entries. 
%Thus, our guarantees are stronger than those  in \cite{Taka_14} and \cite{Taka_14b} in which the true matrices were required to be entrywise strictly bounded away from zero. See  Section \ref{sec:related}.
\end{remark}
%Next, we consider   nonseparable $h$-divergences. By the convexity (or concavity) of $h(\cdot,t)$, \eqref{eq:general_div} is a difference-of-convex (DC) function. Therefore, by using either a first-order Taylor expansion or Jensen's inequality, the nonseparable case can be easily converted to the separable case. Such standard techniques are well-studied in the literature~\cite{Mairal_15,Yang_11c}.
%\subsection
%\noindent 
\subsection{A Concrete Example} %: Multiplicative Updates for the $\alpha$-Divergences}
As the first-order surrogate function~\eqref{eq:first_order_surr_W} in Proposition~\ref{proposition:construction} and the multiplicative update rule~\eqref{eq:MU_W} in Proposition~\ref{proposition:deriv_MU} may seem abstract, as a concrete example, 
%To better illustrate our general multiplicative updates in \eqref{eq:MU_W}, 
we apply them to the family of $\alpha$-divergences ($\alpha \ne  0$). Details are deferred to Appendix~\ref{app:MU_alpha}.  % as a concrete example.%\footnote{Both cases $\alpha\ne 0$ and $\alpha=0$ will be discussed.} The details are deferred to Sections~\ref{sec:MU_alpha} and \ref{sec:dual_KL}. % in \cite{Zhao_16d}. %The multiplicative updates for other $h$-divergences can be similarly derived. 
\subsection{Extension to the Dual KL Divergence }
When $\alpha = 0$, the corresponding $\alpha$-divergence is called the (generalized) dual  KL divergence. Strictly speaking, it does not belong to the class of $h$-divergences.  
However, equipped with a few more   algebraic manipulations based on several technical definitions, we can also construct a first-order surrogate function and derive a multiplicative update in the form of \eqref{eq:first_order_surr_W} and \eqref{eq:MU_W} respectively. See Appendix~\ref{app:dual_KL} for details. Consequently, the result of our convergence analysis (i.e., Theorem~\ref{thm:main}) also applies to this case.  
%as we will show in Section~\ref{sec:dual_KL}, all the %propositions and theorems
%results in this paper also hold for this case %See Section~\ref{sec:MU_alpha} for its definition. We show in  Section~\ref{sec:dual_KL} that 
%(). % See Section~\ref{sec:dual_KL} for details.%} 

%\begin{remark}
%For the purpose of our convergence analysis (see Section~\ref{sec:conv_analysis}), proving the existence of the first-order surrogate functions will suffice. However, in Section~\ref{sec:construct_first_order} and \ref{sec:deriv_MU}, we  explicitly construct the first-order surrogate functions and derive the multiplicative updates. This can be regarded as a {\em side contribution} of our work---the multiplicative updates can be very useful for practitioners to solve real-world problems, based on suitable $h$-divergences and regularizations. 
%\end{remark}

%\begin{algorithm}[t]
%\caption{Multiplicative Updates for the $\alpha$-Divergences} \label{algo:beta}
%\begin{algorithmic} 
%\State {\bf Input}: basis matrix $\bW_{t-1}$, data sample $\bv_t$, coefficient vector $\bh_t^k$, maximum number of iterations $q$
%\State {\bf Initialize} $\alpha\in(0,0.5),\gamma\in(0,1),i:=0,\xi^0:=1$
%\State {\bf while} $\bard_t(\bh_t^k-\xi^i\nabla\bard_t(\bh_t^k))>\bard_t(\bh_t^k)-\alpha\xi^i\normt{\nabla\bard_t(\bh_t^k)}^2$\\
%\hspace{7cm} {\bf and} $i\le q$
%\begin{align*}
%\xi^{i+1} := \gamma\xi^i,\quad i := i+1
%\end{align*}
%\State {\bf end} %some convergence criterion is met
%\State {\bf Output}: Final step size $\beta_t^k\defeq\xi^i$
%\end{algorithmic}
%\end{algorithm}

\section{Convergence Analysis} \label{sec:conv_analysis}
\subsection{Preliminaries} \label{sec:prelim}
We first define important concepts and quantities that will be used in our convergence analysis in Section~\ref{sec:conv_result}. 
\begin{definition}[Stationary points of constrained optimization problems]\label{def:stat_pt}
Given a finite-dimensional real Euclidean space $\calX$ with inner product $\lrangle{\cdot}{\cdot}$, a differentiable function $g:\calX\to\bbR$  and a set $\calK\subseteq\calX$, $x_0\in\calK$ is a stationary point of the constrained optimization problem $\min_{x\in\calK}g(x)$ if $\lrangle{\nabla g(x_0)}{x-x_0}\ge 0$, for all $x\in\calK$.
\end{definition}
Define 
%\begin{equation}
$\bX \defeq \left[\bW^T\;\bH\right]\in\bbR_{+}^{K\times(F+N)}$
%\end{equation}
and with a slight abuse of notation, we write $\ell(\bX) \defeq \ell(\bW,\bH)$.
Thus by Definition~\ref{def:stat_pt}, we have that $(\barbW,\barbH)$ %\in\bbR_+^{F\times K}\times\bbR_+^{K\times N}$  %$\barbX \defeq [\barbW^T\;\barbH]$ 
is a stationary point of \eqref{eq:new_opt} if and only if $\lrangle{\nabla_\bX \ell(\barbX)}{\bX-\barbX}\ge 0$, for any $\bX\in\bbR_{+}^{K\times(F+N)}$, where $\barbX \defeq [\,\barbW^T\;\barbH\,]$. In particular, this is true if %Therefore a sufficient condition for the stationarity of $(\barbW,\barbH)$ is
\begin{align}
\lrangle{\nabla_\bW \ell(\barbW,\barbH)}{\bW-\barbW} \ge 0, &\;\;\forall\,\bW\in\bbR_+^{F\times K},\label{eq:stat_W}\\
\lrangle{\nabla_\bH \ell(\barbW,\barbH)}{\bH-\barbH} \ge 0, &\;\;\forall\,\bH\in\bbR_+^{K\times N}.\label{eq:stat_H}
\end{align}
\begin{remark}
In some previous works (e.g., \cite{Lin_07c}),   stationary points are defined in terms of KKT conditions, i.e.,%\footnote{Here we use $\nabla_\bW\ell(\barbW,\barbH)$ and $\nabla_\bH\ell(\barbW,\barbH)$ to denote $\nabla_\bW\ell(\bW,\barbH)\big\vert_{\bW=\barbW}$ and $\nabla_\bH\ell(\barbW,\bH)\big\vert_{\bH=\barbH}$ respectively.} %for any $(i,j,k)\in[F]\times[N]\times[K]$,
\begin{align}
%\barw_{ik}\ge 0, \hspace{.5cm}&\hspace{-.5cm}\;\; \barh_{kj}\ge 0,\\
%\frac{\partial}{\partial w_{ik}} \ell(\bW,\barbH)\Big\vert_{\bW=\barbW} \ge 0, \hspace{.5cm}&\hspace{-.5cm}\;\; \frac{\partial}{\partial h_{kj}} \ell(\barbW,\bH)\Big\vert_{\bH=\barbH} \ge 0,\\
%\left(\frac{\partial}{\partial w_{ik}} \ell(\bW,\barbH)\Big\vert_{\bW=\barbW}\right)\barw_{ik} = 0,\hspace{-2.5cm}& \\
%\left(\frac{\partial}{\partial h_{kj}} \ell(\barbW,\bH)\Big\vert_{\bH=\barbH}\right)\barh_{kj}=0.\hspace{-2.5cm}&
\barbW\ge 0, &\; \barbH\ge 0\nn\\
\nabla_\bW\ell(\barbW,\barbH)\ge 0, &\; \nabla_\bH\ell(\barbW,\barbH)\ge 0\nn\\
\barbW \odot \nabla_\bW\ell(\barbW,\barbH)\big\vert_{\bW=\barbW}= 0, &\;\barbH \odot \nabla_\bH\ell(\barbW,\barbH)\big\vert_{\bH=\barbH}= 0. \nn%\label{eq:KKT3}
\end{align}
Since both $\barbW$ and $\barbH$ are nonnegative, it is easy to show these three conditions are equivalent to \eqref{eq:stat_W} and \eqref{eq:stat_H}. In our analysis, we will use \eqref{eq:stat_W} and \eqref{eq:stat_H} for   convenience. 
\end{remark}
\begin{definition}[Convergence of a sequence to a set]\label{def:conv_set}
Given a finite-dimensional real Euclidean space $\calX$ with norm $\norm{\cdot}$, a sequence $\{x_n\}_{n=1}^\infty$ in $\calX$ is said to converge to a set $\calA\subseteq\calX$, denoted as $x_n\to\calA$, if $\lim_{n\to\infty} \inf_{a\in\calA}\norm{x_n-a}= 0$.
\end{definition}
%The proof of the following result can be found in Bertsekas \cite{Bert_99}.

\subsection{Main Result}\label{sec:conv_result} %\vspace{-.05in}
\begin{theorem}\label{thm:main}
For any $\bV\in\bbR_{+}^{F\times N}$, $K\in\bbN$, $\lambda_1,\tlambda_1 > 0$ and $\lambda_2,\tlambda_2 \ge  0$, the sequence of iterates $\{(\bW^t,\bH^t)\}_{t=1}^\infty$ generated by Algorithm~\ref{algo:general} converges to the set of stationary points of \eqref{eq:new_opt}.
\end{theorem}
\begin{proof}
Since it is known that $x_n\! \to \!\calA$ (cf.\ Definition \ref{def:conv_set}) if and only if every limit point of $\{x_n\}_{n=1}^\infty$ lies in $\calA$,  
%First, by Lemma~\ref{lem:equiv_conv_set}, 
it suffices to show every limit point of  $\{(\bW^t,\bH^t)\}_{t=1}^\infty$ is a stationary point of \eqref{eq:new_opt}.
Since $\lambda_1,\tlambda_1\!>\!0$, $(\bW,\bH)\mapsto\ell(\bW,\bH)$ is jointly coercive~\cite{Bert_99} in $(\bW,\bH)$. In addition, the continuous differentiability of $h(\cdot,t)$ %(see Lemma~\ref{lem:reg_h}) 
implies the joint continuous differentiability of $(\bW,\bH)\mapsto\ell(\bW,\bH)$ in $(\bW,\bH)$. %\textcolor{red}{why? i think we only need continuity for $\calS_0$ to be compact}. 
Hence the sub-level set
\begin{align}
\calS_0 &\defeq \Big\{(\bW,\bH)\in\bbR_+^{F\times K}\!\times\!\bbR_+^{K\times N}\;\big\vert\; \nn\\*
 &\qquad \ell(\bW,\bH)\le \ell(\bW^0,\bH^0)\Big\}
\end{align}
is compact. Since the sequence $\{\ell(\bW^t,\bH^t)\}_{t=1}^\infty$ is nonincreasing, %(and hence convergences to a unique limit), 
$\{(\bW^t,\bH^t)\}_{t=1}^\infty\subseteq\calS_0$. By the compactness of $\calS_0$, $\{(\bW^t,\bH^t)\}_{t=1}^\infty$ has at least one limit point. Pick any such limit point and denote it as $\cbZ\!\defeq\!(\cbW,\cbH)$.
We also define 
\begin{align}
\bZ^t \defeq\left\{\hspace{-.2cm}\begin{array}{ll}
\left(\bW^{t/2},\bH^{t/2}\right), & t \mbox{ even}\\
\left(\bW^{\floor{t/2}+1},\bH^{\floor{t/2}}\right), & t \mbox{ odd}
\end{array}\right.\hspace{-.2cm}, \;\forall\,t\in\bbN. %\mbox{ and }\; \cbZ\defeq \left(\cbW,\cbH\right).\nn
\end{align} 
Note that the subsequence of $\{\bZ^t\}_{t=1}^\infty$ with even indices, i.e., $\{\bZ^{2t'}\}_{t'=1}^\infty$ correspond to the sequence $\{(\bW^t,\bH^t)\}_{t=1}^\infty$.
Hence, there exists a   subsequence $\left\{\bZ^{t_j}\right\}_{j=1}^\infty$ that converges to $\cbZ\in\calS_0$ and $\{t_j\}_{j=1}^\infty$ are all even. % \textcolor{red}{why all even? what if the subsequence's indices are all odd}. 
Moreover, there exists a  subsequence of the sequence $\left\{\bZ^{t_j-1}\right\}_{j=1}^\infty$, denoted as $\left\{\bZ^{t_{j_i}-1}\right\}_{i=1}^\infty$, such that $\bZ^{t_{j_i}-1}$ converges to (possibly) some other limit point $\cbZ'\defeq(\cbW',\cbH')$ as $i\to\infty$.

Next we show $\cbZ=\cbZ'$. By the update rule \eqref{eq:upd_H}, we have
\begin{equation}
\bH^{t_{j_i}/2} \in \argmin_{\bH\in\bbR_{+}^{K\times N}} G_2\left(\bH\vert\bZ^{t_{j_i}-1}\right), \,\forall\,i\in\bbN.
\end{equation}
Thus for any $i\in\bbN$,
\begin{equation}
G_2(\bH^{t_{j_i}/2}\vert\bZ^{t_{j_i}-1}) \le G_2(\bH\vert\bZ^{t_{j_i}-1}), \quad\forall\,\bH\in\bbR_{+}^{K\times N}. \label{eq:G2_le}
\end{equation}
By \ref{prop:zero_ge}, we also have for any $i\in\bbN$,
\begin{equation}
\ell(\bZ^{t_{j_i}/2})\defeq\ell(\bW^{t_{j_i}/2},\bH^{t_{j_i}/2})\le G_2(\bH^{t_{j_i}/2}\vert\bZ^{t_{j_i}-1}).\label{eq:G2_ge}
\end{equation}
Taking $i\to\infty$ on both sides of \eqref{eq:G2_le} and \eqref{eq:G2_ge}, we have
\begin{equation}
\ell(\cbZ)\le G_2(\cbH\vert\cbZ') \le G_2(\bH\vert\cbZ'), \quad\forall\,\bH\in\bbR_{+}^{K\times N}, \label{eq:G2_lege}
\end{equation}
by the joint continuity of $G_2(\cdot\vert\cdot)$ in both arguments in \ref{prop:mu}. 
Thus%\vspace{-.2cm}
\begin{equation}
\cbH \in  \argmin_{\bH\in\bbR_{+}^{K\times N}} G_2(\bH\vert\cbZ'). \label{eq:min_cbH}
\end{equation}
Taking $\bH=\cbH'$ in \eqref{eq:G2_lege}, we have
%\begin{equation}
%\ell(\cbZ)\le G_2(\cbH\vert\cbZ')\le G_2(\bH^{t_i/2-1}\vert\cbZ'),\,\forall\,i\in\bbN.
%\end{equation}
%Again by taking $i\to\infty$, we have
\begin{equation}
\ell(\cbZ)\le G_2(\cbH\vert\cbZ')\le G_2(\cbH'\vert\cbZ')\defeq \ell(\cbZ').
\end{equation}
Since $\{\ell(\bZ^t)\}_{t=1}^\infty$ converges (to a unique limit point), we have $\ell(\cbZ)=\ell(\cbZ')$. This implies  that $\ell(\cbZ)= G_2(\cbH\vert\cbZ')$. Then for any $\bH\in\bbR_{+}^{K\times N}$,
\begin{equation}
\hspace{.0cm}G_2(\cbH'\vert\cbZ')= \ell(\cbZ')=\ell(\cbZ)= G_2(\cbH\vert\cbZ')\le G_2(\bH\vert\cbZ'). \label{eq:G2_ell_chain}%\,\forall\,\bH\ge 0.
\end{equation}
This implies  that %\vspace{-.3cm}
\begin{equation}
%\vspace{-.2cm}
\cbH' \in \argmin_{\bH\in\bbR_{+}^{K\times N}} G_2(\bH\vert\cbZ').\label{eq:min_cbHprime}
\end{equation}
Combining \eqref{eq:min_cbH} and \eqref{eq:min_cbHprime}, by the strictly convexity of $G_2(\cdot\vert \cbZ')$ in \ref{prop:strict_convex}, $\cbH=\cbH'$. By symmetry, we can show $\cbW=\cbW'$, hence $\cbZ=\cbZ'$. Thus \eqref{eq:G2_ell_chain} becomes
\begin{equation}
G_2(\cbH\vert\cbZ) \le G_2(\bH\vert\cbZ), \,\forall\,\bH\in\bbR_{+}^{K\times N}. \label{eq:min_G2_cbH}
\end{equation}

Now, the convexity of $G_2(\cdot\vert\cbZ)$ implies that
\begin{equation}
\lrangle{\nabla_\bH G_2(\cbH\vert\cbZ)}{\bH-\cbH} \ge 0, \,\forall\,\bH\in\bbR_{+}^{K\times N}.
\end{equation}
From the first-order property of $G_2(\cdot\vert\cbZ)$ in \ref{prop:first_eq}, we have%\vspace{-.2cm} % implies
\begin{equation}
%\vspace{-.1cm}
\lrangle{\nabla_\bH\ell(\cbW,\cbH)}{\bH-\cbH} \ge 0, \,\forall\,\bH\in\bbR_{+}^{K\times N}. \label{eq:variat_H}
\end{equation}
Similarly, we also have%\vspace{-.2cm}
\begin{equation}
%\vspace{-.1cm}
\lrangle{\nabla_\bW\ell(\cbW,\cbH)}{\bW-\cbW} \ge 0, \,\forall\,\bW\in\bbR_{+}^{F\times K}. \label{eq:variat_W}
\end{equation}
The variational inequalities \eqref{eq:variat_H} and \eqref{eq:variat_W} together show that $(\cbW,\cbH)$ is a stationary point of \eqref{eq:new_opt}. 
\end{proof}
%\vspace{-.1cm}
\begin{remark}
%The intuitions of our proof are explained as follows. 
We now provide some intuitions of the proof.  We first use the positivity of $\lambda_1$ and $\tlambda_1$  to assert that $(\bW,\bH)\mapsto \ell(\bW,\bH)$ is coercive and hence that $\calS_0$ is compact. This allows us to extract convergent subsequences.
The most crucial step   \eqref{eq:min_G2_cbH}  states that at an arbitrary limit point of $\{\bZ^t\}_{t=1}^\infty$, denoted as $\cbZ=(\cbW,\cbH)$, $\cbH$ serves as a minimizer of $G_2(\cdot\vert\cbZ)$ over $\bbR_+^{K\times N}$. By symmetry, $\cbW$  also serves as a minimizer of $G_1(\cdot\vert\cbZ)$ over $\bbR_+^{F\times K}$. In the single-block case, this idea is fairly intuitive. However, to prove \eqref{eq:min_G2_cbH} in the double-block case, we consider {\em two} subsequences $\{\bZ^{t_{j_i}}\}_{i=1}^\infty$ and $\{\bZ^{t_{j_i}-1}\}_{i=1}^\infty$. In each sequence, only $\bW$ or $\bH$ is updated. 
%which correspond to the sequence of updating $\bH$ and updating $\bW$ respectively. 
Then we show these two sequences converge to the same limit point. This implies the  Gauss-Seidel minimization  procedure~\cite[Section~7.3]{Burden_16}  %\textcolor{red}{do you have my Bertsekas book? find the section that Gauss-Seidel is mentioned and cite it more precisely here.} 
in the double-block case is essentially the same as the minimization in the single-block case. The claim then  follows. % immediately. 
%In our proof, without loss of generality, we take $(\cbW,\cbH)$
% To prove \eqref{eq:min_G2_cbH}, we 
\end{remark}

\section{Conclusion and Future Work}
In this work, we present a unified MU algorithm for (weighted) $h$-divergences with $\ell_1$ and Tikhonov regularization and analyze its convergence (to stationary points). 

%\textcolor{red}
{In the future, we plan to investigate the further properties  of the MU algorithm. Specifically, we would like to understand whether it is able to converge to {\em second-order} stationary points~\cite{Nest_06}. This question is motivated by  some recent works} which have shown that for low-rank matrix factorization problems~\cite{Sun_16}, under mild conditions, all the second-order stationary points are local minima. For these problems, the local minima have  been shown to possess strong theoretical properties~\cite{Ge_16}. Therefore, investigation into  convergence to the second-order stationary points of the MU algorithm  is meaningful. 

%\subsection*{Acknowledgements}
{\em Acknowledgements}: The authors would like to thank Prof.\ Zhirong Yang for many useful comments on the manuscript.

%\section{Appendices}
\appendices
%\subsection{Proof of Proposition~\ref{propsition:construction}}\label{sec:proof_prop1}
\setcounter{lemma}{0}
\renewcommand{\thelemma}{T-\arabic{lemma}}

\section{Proof of Proposition~\ref{proposition:construction}}\label{app:proof_construct}
We first show that $G_1(\cdot\vert\cdot)$ is a surrogate function of $(\bW,\bH) \mapsto \ell(\bW,\bH)$ for $\bW$ by %verifying properties 
 decomposing $G_1(\cdot\vert\cdot)$ into two functions $\tilG_1(\cdot\vert\cdot)$ and $\barG_1(\cdot\vert\cdot)$, where  
\begin{align}
&\tilG_1(\bW\vert\tilbZ) \defeq \sum_{i=1}^F\sum_{k=1}^K \left[s_{ik}^+\tilw_{ik}\, h\left(\frac{w_{ik}}{\tilw_{ik}},\vartheta_{2}\right)\right.\nn\label{eq:surr_sep}\\
&\hspace{4cm}\left.-s_{ik}^-\tilw_{ik}\, h\left(\frac{w_{ik}}{\tilw_{ik}},\vartheta_{1}\right)\right],\\
&\barG_1(\bW\vert\tilbZ) \defeq \sum_{i=1}^F\sum_{k=1}^K (\lambda_1\tilw_{ik} + 2\lambda_2\tilw_{ik}^2)  h\left(\frac{w_{ik}}{\tilw_{ik}},\vartheta_{2}\right).
\end{align}
%Next we 
Define constants $\{\zeta'_p\}_{p=1}^P$ such that $\zeta'_p\defeq 1$ if $\zeta_p \in (0,1)$ and $\zeta'_p\defeq \zeta_p$ otherwise, for any $p \in [P]$. Accordingly, define 
\begin{equation}
\zeta'_{\min} \defeq \min\{\zeta'_p\}_{p=1}^P \quad \mbox{and}\quad \zeta'_{\max} \defeq \max\{\zeta'_p\}_{p=1}^P. \label{eq:zeta_max_min_prime}
\end{equation} When $\vartheta_{1} = \zeta'_{\min}$ and $\vartheta_{2} = \zeta'_{\max}$, Yang and Oja~\cite{Yang_11c} showed that  $\tilG_1(\cdot\vert\cdot)$ is a surrogate function of $(\bW,\bH) \mapsto D(\bV\Vert\bW\bH)$ for $\bW$. By Lemmas \ref{lem:reg_h} and~\ref{lem:cal_surr}\ref{item:surr_trans}, $\tilG_1(\cdot\vert\cdot)$ is a surrogate function for any $\vartheta_2\ge \zeta'_{\max}$. % (up to an additive constant that is independent of $\bW$).
Define a new function %$\barG'_1(\cdot\vert\cdot)$ such that 
\begin{equation}
\barG'_1(\bW\vert\tilbZ)\defeq\barG_1(\bW\vert\tilbZ) + \sum_{i=1}^2 \lambda_i \phi_i(\tilbW)  +  \sum_{j=1}^2\tlambda_j \phi_j(\tilbH),
\end{equation} 
so that $\bW \mapsto \barG'_1(\bW\vert\tilbZ)$ and $\bW \mapsto \barG'_1(\bW\vert\tilbZ)$ differs by a constant that is independent of $\bW$. 
By Lemma~\ref{lem:cal_surr}\ref{item:surr_sum}, to show $G_1(\cdot\vert\cdot)$ is %a first-order surrogate functon of 
is a first-order surrogate function of $(\bW,\bH) \mapsto \ell(\bW,\bH)$ for $\bW$, it suffices to show $\barG'_1(\cdot\vert\cdot)$ is a surrogate function of $(\bW,\bH) \mapsto \sum_{i=1}^2 \lambda_i \phi_i(\bW)  +  \sum_{j=1}^2\tlambda_j \phi_j(\bH)$ for $\bW$, with $\vartheta_2\ge \zeta'_{\max}$.   
%To show that it satisfies~\ref{prop:zero_eq} and~\ref{prop:zero_ge}, we have
%In addition, 
First, note that 
\begin{align}%\flushleft
&\hspace{-.4cm}\quad\lambda_1\norm{\bW}_{1,1} + \lambda_2\norm{\bW}_\rmF^2\nn\\
&\hspace{-.4cm}= \sum_{i=1}^F\sum_{k=1}^K \lambda_1 \tilw_{ik} \left[\frac{w_{ik}}{\tilw_{ik}} - 1\right] + 2\lambda_2\tilw_{ik}^2\left[\frac{1}{2}\left(\frac{w_{ik}}{\tilw_{ik}}\right)^2 - \frac{1}{2}\right]\nn\\
& \hspace{4cm} +\lambda_1\normt{\tilbW}_{1,1} + \lambda_2\normt{\tilbW}_\rmF^2\\
&\hspace{-.4cm}\le \sum_{i=1}^F\sum_{k=1}^K (\lambda_1\tilw_{ik}  +  2\lambda_2\tilw_{ik}^2)  h\left(\frac{w_{ik}}{\tilw_{ik}},\max\{\zeta'_{\max},1,2\sgn(\lambda_2)\} \right)\nn\\
&\hspace{4cm}+\lambda_1\normt{\tilbW}_{1,1} + \lambda_2\normt{\tilbW}_\rmF^2,\label{eq:upperbound_12}
\end{align}
where $\sgn(\lambda_2)$ equals $0$ if $\lambda_2 = 0$ and equals $1$ if $\lambda_2 > 0$, and in~\eqref{eq:upperbound_12}  we use the monotonicity of $h(\sigma,\cdot)$ for any $\sigma > 0$ in Lemma~\ref{lem:reg_h}. 
Since $h(1,t) = 0$ for any $t > 0$, by choosing $\vartheta_2 = \max\{\zeta'_{\max},1,2\sgn(\lambda_2)\}$,
% and using Lemma~\ref{lem:cal_surr}\ref{item:surr_trans}, 
 we see that
$\barG'_1(\cdot\vert\cdot)$ satisfies properties~\ref{prop:zero_eq} and~\ref{prop:zero_ge}. 
In addition, by~\eqref{eq:deriv_h} in Lemma~\ref{lem:reg_h}, $\barG'_1(\cdot\vert\cdot)$ obviously satisfies~\ref{prop:mu}. 
%$\zeta_{\max}\defeq\max\{\zeta'_{\max},1,2\sgn(\lambda_2)\}$, $\zeta_{\min}\defeq\zeta'_{\min}$ and 
%in \eqref{eq:upperbound_12} we make use of the monotonicity of $h(\nu,\cdot)$ in Lemma~\ref{lem:reg_h}.

To prove $G_1(\cdot\vert\cdot)$ is a first-order surrogate function of $(\bW,\bH) \mapsto \ell(\bW,\bH)$ for $\bW$, we show it also satisfies properties~\ref{prop:first_eq} and~\ref{prop:strict_convex}. 
  %$G(\bW\vert\tilbW,\tilbH)$ is a surrogate function, first notice that by the differentiability of $h(\cdot,t)$ in Lemma~\ref{lem:reg_h}, it obviously satisfies \ref{prop:mu}. Thus by Lemma~\ref{lem:cal_surr}, it suffices to show 
%\begin{align}
%&\quad\lambda_1\norm{\bW}_{1,1} + \lambda_2\norm{\bW}_F^2\\
%&\eqcst \sum_{i=1}^F\sum_{k=1}^K \lambda_1 \tilw_{ik} \left[\frac{w_{ik}}{\tilw_{ik}}\right] + 2\lambda_2\tilw_{ik}^2\left[\frac{1}{2}\left(\frac{w_{ik}}{\tilw_{ik}}\right)^2\right]\\
%&\lecst \sum_{i=1}^F\sum_{k=1}^K (\lambda_1\tilw_{ik} + 2\lambda_2\tilw_{ik}^2)  h\left(\frac{w_{ik}}{\tilw_{ik}},\zeta_{\max}\right),\label{eq:upperbound_12}
%\end{align}
%where $\zeta_{\max}\defeq\max\{\zeta'_{\max},1,2\sgn(\lambda_2)\}$, $\zeta_{\min}\defeq\zeta'_{\min}$ and 
%in \eqref{eq:upperbound_12} we make use of the monotonicity of $h(\nu,\cdot)$ in Lemma~\ref{lem:reg_h}.
First, for any $i \in [F]$ and $k \in [K]$, %we have 
\begin{align}
& \left[\nabla_\bW G_1(\bW\vert\tilbZ)\right]_{ik}  \nn\\*
&=  (s_{ik}^++\lambda_1+2\lambda_2\tilw_{ik})\left(\frac{w_{ik}}{\tilw_{ik}}\right)^{\vartheta_2-1}-s_{ik}^-\left(\frac{w_{ik}}{\tilw_{ik}}\right)^{\vartheta_1-1}.\nn%\; \forall\,(i,k)\in[F]\times[K].\label{eq:grad_surr}
\end{align}
Therefore, 
%From \eqref{eq:grad_surr}, we have 
%\begin{equation}
$\nabla_\bW G_1(\bW\vert\tilbZ)\big\vert_{\bW=\widetilde{\bW}} = \nabla_\bW \ell(\bW,\tilbH)\big\vert_{\bW=\widetilde{\bW}}$.
%\end{equation}
%Finally, to prove \ref{prop:strict_convex}, it suffices to show for any $\bW\in\bbR_{++}^{F\times K}$,
%\begin{equation}
%\frac{\partial^2}{\partial w_{ik}^2}G(\bW\vert\tilbW,\tilbH)>0,\,\forall\,(i,k)\in[F]\times[K].
%\end{equation}
%Next, by the definition of $\{\zeta'_p\}_{p=1}^P$, we have $\zeta'_{\min}\le 1\le\zeta'_{\max}$
This shows~\ref{prop:first_eq}. Next, since $\vartheta_1 \le 1 \le\vartheta_2$, $\vartheta_2-\vartheta_1 > 0$ and $\bS^+,\bS^- \in \bbR_{+}^{F\times K}$, %we have that 
for any $i \in [F]$, $k \in [K]$ and $\bW \in \bbR_{++}^{F\times K}$,
\begin{align}
&\frac{\partial^2}{\partial w_{ik}^2}G_1(\bW\vert\tilbZ) = (1-\vartheta_1)s^-_{ik}\left(\frac{w_{ik}}{\tilw_{ik}}\right)^{\vartheta_1-2}\nn\\
&\quad\quad+\left(\frac{s^+_{ik}+\lambda_1}{\tilw_{ik}}+2\lambda_2\right)\left(\vartheta_2-1\right)\left(\frac{w_{ik}}{\tilw_{ik}}\right)^{\vartheta_2-2}>0.\nn
\end{align}
This implies $\bW \mapsto G_1(\bW\vert\tilbZ)$ is strictly convex on $\bbR_{++}^{F\times K}$. Since $\bW \mapsto G_1(\bW\vert\tilbZ)$ is continuous and convex on $\bbR_{+}^{F\times K}$, it is strictly convex on $\bbR_{+}^{F\times K}$. This proves \ref{prop:strict_convex}.

\section{First-order Surrogate Functions and Multiplicative Updates for the $\alpha$-Divergences $(\alpha\ne 0)$}\label{app:MU_alpha}
%For ease of reference, we first define the $\alpha$-divergences.
\begin{definition}[$\alpha$-divergences]\label{def:alpha_div}
Given any matrix $\bV\in\bbR_{++}^{F\times N}$, the $\alpha$-divergences $D^{\rm alp}_\alpha(\bV\Vert\cdot) : \bbR_{++}^{F\times N} \to \bbR$, % (parameterized by $\alpha  \in \bbR$), 
is defined as $D^{\rm alp}_\alpha(\bV\Vert\hatbV) \defeq  \sum_{i=1}^F\sum_{j=1}^N d^{\rm alp}_\alpha(v_{ij}\Vert \hatv_{ij})$, where for $v,\hatv > 0$, %positive $v$ and $\hatv$,
\[
d^{\rm alp}_\alpha(v\Vert\hatv)  \defeq  \left\{\hspace{-.2cm}\begin{array}{ll}\dfrac{\left(\hatv\left[\left(v/\hatv\right)^\alpha - 1\right] - \alpha(v-\hatv)\right)}{\alpha(\alpha-1)}, &\hspace{-.2cm}\alpha \in  \bbR \setminus \{0,1\}\\
v\log(v/\hatv)-v+\hatv, &\hspace{-.2cm} \alpha= 1\\
\hatv\log(\hatv/v)-\hatv+v, &\hspace{-.2cm} \alpha= 0\\
\end{array}\hspace{-.0cm}.\right.
\]
\end{definition}
%In this section, we focus on the case $\alpha\ne 0$. 
To construct the first-order surrogate function $G^{{\rm alp},\alpha}_1(\cdot\vert\cdot)$ of $(\bW,\bH) \mapsto \ell(\bW,\bH)$ in~\eqref{eq:new_obj} for $\bW$, when $D(\bV\Vert\cdot)$ belongs to the family of $\alpha$-divergences ($\alpha \ne 0$), first recall the definitions of $\bS^+$ and $\bS^-$ in Proposition~\ref{proposition:construction}. 
From the definition of the $\alpha$-divergences (in Definition~\ref{def:alpha_div}), given any $\tilbW \in \bbR_+^{F \times K}$ and $\tilbH \in \bbR_+^{K \times N}$, and for any $(i,k)\in[F]\times[K]$, we have %when $\alpha\ne 0$, %of the $\alpha$
\begin{align}
s^+_{ik} = \frac{1}{\alpha}\sum_{j=1}^N\tilh_{kj},\;\;\mbox{and}\;\;
s^-_{ik} = \frac{1}{\alpha}\sum_{j=1}^Nq_{ij}^\alpha \tilh_{kj}, \;\;\alpha\ne 0,
%\left\{\hspace{-.2cm}\begin{array}{ll}
%\frac{1}{\alpha}\sum_{j=1}^Nq_{ij}^\alpha \tilh_{kj},&\alpha\ne 0\vspace{.2cm}\\
%\frac{1}{\alpha}\sum_{j=1}^N\log(q_{ij}) \tilh_{kj},&\alpha=0
%\end{array}\right.\nn,
\end{align}
where for any $(i,j) \in [F] \times [N]$, $q_{ij} \defeq {v_{ij}}/{(\tilbW\tilbH)_{ij}}$. In addition, from Definition~\ref{def:alpha_div}, we can also observe the values of $\zeta_1$ and $\zeta_2$ (see Definition~\ref{def:general_div}), hence deduce from~\eqref{eq:zeta_max_min_prime} that  
\begin{align}
(\zeta'_{\max},\zeta'_{\min}) = \left\{\hspace{-.2cm}\begin{array}{ll}
%(1,0), & \alpha=1\\
(1,1-\alpha), & \alpha>0\\ %\in\bbR_{++}\\%\setminus\{1\}\\
%(1+\eta,1), & \alpha=0\\
(1-\alpha,1), & \alpha <0
\end{array}\right..\label{eq:zeta_max_min}
\end{align}
By choosing $\vartheta_{1} = \zeta'_{\min}$ and $\vartheta_2 = \max\{\zeta'_{\max},1,2\sgn(\lambda_2)\}$ as in Appendix~\ref{app:proof_construct}, we can obtain $G_1^{{\rm alp},\alpha}(\cdot\vert\cdot)$ per~\eqref{eq:first_order_surr_W} in Proposition~\ref{proposition:construction}. 
%In \cite{Cichoc_08,Yang_11c}, the surrogate function (up to some additive constants) of $(\bW,\bH)\mapsto D(\bW\Vert\bH)$ for $\bW$ at $(\tilbW,\tilbH)$ with the $\alpha$-divergences, $\barG_\alpha(\bW\vert\tilbW,\tilbH)$ is given by\vspace{.2cm}\\
%\begin{subequations}
%$\barG_\alpha(\bW\vert\tilbW,\tilbH) =$ \\[-.1cm]
%\begin{empheq}[left={\empheqlbrace}]{align}
%&\sum_{ik}\frac{1}{\alpha(1-\alpha)}\sum_{j} (1-\alpha)q_{ij}^{-1}({w_{ik}}/{\tilw_{ik}})\nn\\[-.1cm]
%&\hspace{1cm} - q_{ij}^{\alpha-1}\left({w_{ik}}/{\tilw_{ik}}\right)^{1-\alpha},\;\alpha\in\bbR\setminus\{0,1\}\vspace{.1cm},\label{eq:G_alpha_ne0}\\
%&\sum_{ik}\sum_{j}q_{ij}^{-1}({w_{ik}}/{\tilw_{ik}})-\nn\\[-.3cm]
%&\hspace{1.4cm}\log\left(q_{ij}^{-1}({w_{ik}}/{\tilw_{ik}})\right),\;\alpha=1.
%\end{empheq}
%\end{subequations}
%Thus %by \eqref{eq:zeta_max_prime}, we have
%Therefore, $\zeta_{\max}$ and $\zeta_{\min}$ can be obtained by their definitions. 
Additionally, 
the multiplicative update in \eqref{eq:MU_W} in the case of $\alpha$-divergences becomes
\begin{align*}\hspace{-.1cm}
\begin{array}{ll}
w_{ik} := \tilw_{ik}\left(\dfrac{\sum_{j=1}^Nq_{ij}^\alpha \tilh_{kj}}{\sum_{j=1}^N\tilh_{kj}+2\alpha\lambda_2\tilw_{ik}+\alpha\lambda_1}\right)^{1/\phi(\alpha)}  ,&   \alpha\ne 0
%w_{ik} := \tilw_{ik}\left(\dfrac{\sum_{j=1}^N\log(q_{ij}) \tilh_{kj}}{\sum_{j=1}^N\tilh_{kj}+2\alpha\lambda_2\tilw_{ik}+\alpha\lambda_1}\right)^{1/\eta},&   \alpha= 0
\end{array}  ,
\end{align*}
where 
\begin{align}
\phi(\alpha) = \left\{\hspace{-.2cm}
\begin{array}{ll}
\alpha+\sgn(\lambda_2), & \alpha>0\\
1, & \alpha\in(-1,0)\\
-\alpha,&\alpha<-1
\end{array}
\right..
\end{align}

\section{First-order Surrogate Functions and Multiplicative Updates for the Dual KL Divergence} \label{app:dual_KL}
In~\cite{Cichoc_10}, a surrogate function of $(\bW,\bH) \mapsto D(\bV\Vert\bW\bH)$  for $\bW$, when $D(\bV\Vert\cdot)$ is the dual KL divergence, is given by %denoted as $G^{\rm alp}_0()$ %by taking $\alpha\to 0$ and using L'H\^opital's rule, we have
\begin{align}
&\tilG^{{\rm alp},0}_1(\bW\vert\tilbZ) \defeq \sum_{ik} w_{ik}\log\left(\frac{w_{ik}}{\tilw_{ik}}\right)\sum_{j} h_{kj}\nn \\
&\hspace{3.5cm}- w_{ik} \sum_j\left(\log q_{ij}+1\right)h_{kj}.
\end{align}
However, this surrogate function {cannot be expressed as~\eqref{eq:surr_sep} with an appropriate choice of parameters}, hence it cannot be directly used to derive the first-order surrogate function for~\eqref{eq:new_obj}. Therefore {consider a majorant for $\tilG^{{\rm alp},0}_1(\bW\vert\tilbZ)$ as follows}
\begin{align}
&\hspace{-.3cm}{\tilG^{{\rm alp},0'}_1}(\bW\vert\tilbZ) \defeq\frac{1}{\eta}\sum_{ik}\left({w_{ik}}/{\tilw_{ik}}\right)^{1+\eta}\tilw_{ik}\sum_{j}q_{ij}^{-\eta}h_{kj} \nn\\
&\hspace{2.7cm} - (1+\eta)(w_{ik}/\tilw_{ik})\tilw_{ik}\sum_{j}h_{kj}, \label{eq:barG}
\end{align}
where $\eta$ is any positive real number. 
By Lemma~\ref{lem:cal_surr}\ref{item:surr_trans}, $\tilG^{{\rm alp},0}_1(\cdot\Vert\cdot)$ is also a surrogate function of $(\bW,\bH) \mapsto D(\bV\Vert\bW\bH)$  for $\bW$. 
By comparing \eqref{eq:surr_sep} {to} \eqref{eq:barG}, we have % observe that %$s^+_{ik}$,  $s^-_{ik}$, $\vartheta_1$ and $\vartheta_2$ in~\eqref{eq:surr_sep} have the following values: 
%The correspondences of $s^+_{ik}$,  $s^-_{ik}$, $\vartheta_1$ and $\vartheta_2$ in \eqref{eq:surr_sep} to those in \eqref{eq:barG} can be easily established:
%\footnote{Note that in this case $s^+_{ik}$ and $s^-_{ik}$ may not represent the sums of positive and unsigned negative terms in $\nabla_\bW D(\bV\Vert\bW\tilbH)\big\vert_{\bW=\widetilde{\bW}}$. However, as long as they lie in $\bbR_{+}^{F\times K}$, all the propositions and theorems in this paper hold.}
\begin{align}
&\vartheta_1 = 1,\;\vartheta_2 \ge 1+\eta,\nn\\
&s^+_{ik} = \frac{1}{\eta}\sum_j q_{ij}^{-\eta}h_{kj}, \;s^-_{ik} = \frac{1+\eta}{\eta}\sum_j h_{kj}.\nn
\end{align}
 Note that in this case $\bS^+$ and $\bS^-$ may not represent the sums of positive and unsigned negative terms in $\nabla_\bW D(\bV\Vert\bW\tilbH)\big\vert_{\bW=\widetilde{\bW}}$. However, as long as they lie in $\bbR_{+}^{F\times K}$, Propositions~\ref{proposition:construction} and~\ref{proposition:deriv_MU} still hold.   
Therefore, by choosing $\vartheta_2 = \max\{1+\eta,2\sgn(\lambda_2)\}$, the first-order surrogate function for the dual KL-divergence can be constructed per Proposition~\ref{proposition:construction}. Additionally, per Proposition~\ref{proposition:deriv_MU}, the multiplicative update is
\begin{equation}
w_{ik} \triangleq w_{ik} \left(\frac{(1+\eta)\sum_j h_{kj}}{\sum_j q_{ij}^{-\eta}h_{kj}+2\eta\lambda_2\tilw_{ik}+\eta\lambda_1}\right)^{1/\psi(\eta)},
\end{equation}
where
\begin{equation}
\psi(\eta) \triangleq \max\{\eta,2\sgn(\lambda_2)-1\}.
\end{equation}
%We remark that in practice, to avoid over-relaxation \textcolor{red}{no idea what's over-relaxation} issues, $\eta$ can be chosen to be small. %This ensures that  $\barG'_0(\cdot\vert\tilbW,\tilbH)$  is a tighter majorant function of $\barG_0(\cdot\vert\tilbW,\tilbH)$.

\section{Technical Lemmas}\label{app:tech_lemma}
Among the following technical lemmas, Lemma~\ref{lem:reg_h} is adapted from {\cite[Lemma~1]{Tan_13a}}, whereas Lemma~\ref{lem:cal_surr}  can be simply proved by definition. 
%We first present two lemmas that will be used in our proof. 
\begin{lemma}[Regularity of $h$ in \eqref{eq:def_h}]\label{lem:reg_h}
For any $t\in\bbR$, denote the natural domain of $h(\cdot,t)$ as $\Xi_t$. Then $h(\cdot,t)$ is continuously differentiable on $\inter(\Xi_t)$, i.e., the interior of $\Xi_t$.  In particular,
\begin{equation}
\frac{\partial}{\partial\sigma} h(\sigma,t) = \sigma^{t-1}, \,\forall\,t\in\bbR,\,\forall\,\sigma\in\inter(\Xi_t).\label{eq:deriv_h}
\end{equation}
In addition, $h(\cdot,t)$ is either convex or concave  on $\inter(\Xi_t)$, for any $t \in \bbR$.
Finally, for every $\nu > 0$, the function $h(\nu,\cdot)$ is nondecreasing on $\bbR$. %In fact, $h(\nu,\cdot)$ is strictly increasing on $\bbR$ unless $\nu=1$. 
\end{lemma}
\begin{lemma}[Calculus of Surrogate Functions]\label{lem:cal_surr}
Let $f:\calX\to\bbR$ be given as in Definition~\ref{def:first_order_surrogate}. Let $i$ be any index in $[n]$. %Then for any $i\in[n]$,
\begin{enumerate}[leftmargin=0.5cm,label=(\alph*)]
\item %For any $i\in[n]$, 
If $f = \barf + \tilf$ and $\barF_i(\cdot\vert\cdot)$ and $\tilF_i(\cdot\vert\cdot)$ are surrogate functions of $(x_1,\ldots,x_n) \mapsto \barf(x_1,\ldots,x_n)$ and $(x_1,\ldots,x_n) \mapsto \barf(x_1,\ldots,x_n)$ for the $x_i$ respectively, then $F_i(\cdot\vert\cdot) \defeq \barF_i(\cdot\vert\cdot)+\tilF_i(\cdot\vert\cdot)$ is a surrogate function of $(x_1,\ldots,x_n) \mapsto f(x_1,\ldots,x_n)$ for $x_i$.\label{item:surr_sum}
\item %For any $i\in[n]$, 
If $F_i(\cdot\vert\cdot)$ is a surrogate function of $(x_1,\ldots,x_n) \mapsto f(x_1,\ldots,x_n)$ for $x_i$, and there exists $\tilF_i:\calX_i\times\calX\to\bbR$ that satisfies \ref{prop:mu} and 
\begin{align}
\tilF_i(\tilx_i\vert\tilx) &= F_i(\tilx_i\vert\tilx), \,\forall\,\tilx\in\calX,\\
\tilF_i(x_i\vert\tilx) &\ge F_i(x_i\vert\tilx), \,\forall\,(x_i,\tilx)\in\calX_i\times\calX,
\end{align}
then $\tilF_i(\cdot\vert\cdot)$ is a surrogate function of $f$ for $x_i$. \label{item:surr_trans}
\end{enumerate}
\end{lemma}

\bibliographystyle{IEEEtran}
\bibliography{ORNMF_ref,RNMF_ref,bregman_ref,stoc_ref,math_opt,stat_ref,dataset}

\end{document}